\documentclass{amsart}
\usepackage[margin=1.5cm]{geometry}
\usepackage{amssymb,amsthm,comment,color,caption,enumerate,adjustbox,enumitem,tikz,
amsmath, amsfonts,amsthm,amssymb,amscd, verbatim, graphicx,color,multirow,booktabs, caption,tikz,tikz-cd, mathdots, todonotes}
\usepackage{colortbl, xcolor}
\usepackage{chngcntr, arydshln, soul, float,subcaption}
\numberwithin{equation}{section}
\usetikzlibrary{arrows,shapes,positioning,decorations.markings}
\usetikzlibrary{matrix, calc, arrows, graphs}
\usepackage{hyperref}
\newtheorem{theorem}{Theorem}[section]
\newtheorem{lemma}[theorem]{Lemma}
\newtheorem{proposition}[theorem]{Proposition}

\newtheorem{corollary}[theorem]{Corollary}

\newtheorem{conjecture}[theorem]{Conjecture}
\counterwithin{table}{section}
\begin{document}
\title[Engel graph]{The Engel graph of almost simple groups}
\author[A. Lucchini]{Andrea Lucchini}
\address{Andrea Lucchini, Dipartimento di Matematica \lq\lq Tullio Levi-Civita\rq\rq,\newline
 University of Padova, Via Trieste 53, 35121 Padova, Italy} 
\email{lucchini@math.unipd.it}
         
\author[P. Spiga]{Pablo Spiga}
\address{Pablo Spiga, Dipartimento di Matematica Pura e Applicata,\newline
 University of Milano-Bicocca, Via Cozzi 55, 20126 Milano, Italy} 
\email{pablo.spiga@unimib.it}
	\maketitle

        \begin{abstract}
Given a finite group $G$,  the Engel graph of $G$ is a directed graph encoding  pairs of elements satisfying some Engel word. From the work of Detomi, Lucchini and Nemmi~\cite{DLN}, the strongly connectivity of the Engel graph of an arbitrary group $G$ is reduced to the understanding of the strongly connectivity of the Engel graph of non-abelian simple groups. 

In this paper, we investigate the strongly connectivity of the Engel graph of finite non-abelian simple groups.
          \end{abstract}
          
          \tableofcontents
\section{Introduction}
Let $\omega$ be a word in the free group of rank $2$ and let $G$ be a group. There is a gadget that gives a combinatorial measure of how often the word $\omega$ is satisfied in $G$. This gadget is the directed graph $\Lambda_\omega(G)$ having vertex set $G$ and where $(x,y)$ is an arc in $\Lambda_\omega(G)$ if and only if $\omega(x,y)=1$. There are some strong reasons to actually study only a certain subgraph of $\Lambda_\omega(G)$. Indeed, let 
\begin{align*}
I_{\textrm{right},\omega}(G)&:=\{g\in G\mid \omega(g,x),\,\forall x\in G\},\\
I_{\textrm{left},\omega}(G)&:=\{g\in G\mid \omega(x,g),\,\forall x\in G\},\\
I_\omega(G)&:=I_{\textrm{right},\omega}(G)\cap I_{\textrm{left},\omega}(G).
\end{align*}
The elements in $I_\omega(G)$ are the vertices of $\Lambda_\omega(G)$ which are in-neighbours and out-neighbours to every other vertex. Therefore, for studying connectedness properties, it is important to introduce the subgraph $\Gamma_\omega(G)$ induced by $\Lambda_\omega(G)$ on $G\setminus I_\omega(G)$.

A  graph in this family is the commuting graph of a group, where $\omega(x,y):=[x,y]=x^{-1}y^{-1}xy$ is the commutator word. Clearly, the commuting graph is undirected because the commutator word is symmetric in the variable $x$ and $y$, but in general $\Gamma_\omega(G)$ is directed. This key feature makes it interesting to study the connectivity and the strongly connectivity of $\Gamma_\omega(G)$.  A directed graph is \textit{\textbf{strongly connected}} if, for any two vertices, there exists a directed path from the first to the second.

In this paper we are interested in two directed graphs, both introduced in~\cite{DLN}, which encode some information on the pairs of elements of $G$ satisfying some Engel word.
Let $x$ and $y$ be free generators of a free group of rank $2$. We define recursively $[x,_0y]:=x$ and $$[x,_{i+1}y]:=[[x,_{i}y],y],$$ for every $i\ge 0$. The word $[x,_ny]$ is the \textit{\textbf{$n^{\mathrm{th}}$ Engel word}}. In this paper we are interested in the $n^{\mathrm{th}}$ \textit{\textbf{Engel graph}} $$\Gamma_{n}(G):=\Gamma_{\omega}(G),$$ where $\omega:=[x,_ny]$ is the $n^{\mathrm{th}}$ Engel word. Clearly, when $n:=1$, we recover the commuting graph of $G$.

 We are also interested in a ``cumulative'' version of $\Gamma_n(G)$. (Let $1={\bf Z}_0(G)\le {\bf Z}_1(G)\le{\bf Z}_2(G)\le \cdots $ be the series of subgroups of $G$, where ${\bf Z}_{n+1}(G)/{\bf Z}_n(G)={\bf Z}(G)$. The subgroup $${\bf Z}_\infty(G):=\bigcup_{n\ge 0}{\bf Z}_n(G)$$ is called the \textit{\textbf{hypercenter}} of $G$.) We let $\Gamma(G)$ be the directed graph having vertex set $G\setminus{\bf Z}_\infty(G)$ and where $(x,y)$ is an arc of $\Gamma(G)$ if and only if $[x,_ny]=1$, for some positive integer $n$. The reason for taking as vertex set $G\setminus{\bf Z}_\infty(G)$ is in~\cite[Introduction]{DLN}.

Theorem~1.1 in~\cite{DLN} shows that $\Gamma(G)$ is always connected (of diameter at most $10$). Therefore, it becomes  interesting to study the strongly connectivity of $\Gamma(G)$. A first investigation on the strongly connectivity of $\Gamma(G)$ is also in~\cite{DLN}. In this context, the main result is the following.
\begin{theorem}[Theorem~1.2,~\cite{DLN}]\label{prel}Suppose that $G/{\bf Z}_\infty (G)$ is not an almost simple
group. Then the Engel graph of $G$ is strongly connected if and only if $G/{\bf Z}_\infty(G)$ is
not a Frobenius group.
\end{theorem}
Clearly, this result demands an investigation on the Engel graph of almost simple groups and this is what we do in this paper.

\begin{theorem}\label{main}
Let $G$ be a non-abelian simple group. Then the Engel graph of $G$ is not strongly connected if and only if $G$ is isomorphic to one of the following groups:
\begin{enumerate}
\item $\mathrm{PSL}_2(q)$ with $q\ge 4$ even or with $q\equiv 5\pmod 8$,
\item ${}^2B_2(q)$ with $q\ge 8$.
\end{enumerate}
\end{theorem}
When $\Gamma(G)$ is strongly connected, an interesting question is determining the smallest $n \in \mathbb N$ such that the subgraph $\Gamma_n(G)$ is already strongly connected. Indeed, we prove that, if $\Gamma$ is strongly connected, then
$\Gamma_3(G)$ is strongly connected, expect when $G=\mathrm{PSL}_2(q)$ and $q\equiv 3\pmod 4$. In this case, $\Gamma_n(G)$ is strongly connected if and only if $n > a$ where $2^a$ is the largest power of $2$ dividing $(q+1)/2$ (see Theorem \ref{thrm:tired}). In particular it follows from Dirichlet's theorem on arithmetic progressions, that for every $n \in \mathbb N,$ there exists a simple group $G$ such that $\Gamma_{n+1}(G)$ is strongly connected, but $\Gamma_n(G)$ is not.

The proof of Theorem~\ref{main} is quite involved and we briefly discuss our approach in Section~\ref{not&prel}. When $G$ is an alternating group or a sporadic simple group, Theorem~\ref{main} is proved in~\cite{DLN}. In particular, our task here is dealing with simple groups of Lie type.

Once that the behaviour of the Engel graph of the finite simple groups is completely understood, it is not difficult to have a more complete result including all the almost simple groups.

\begin{corollary}\label{corcorcor}
Let $G$ be an almost simple group with socle $L$ and with $L<G$. Then $\Gamma_3(G)$ is strongly connected, except
when $G:=\mathrm{Aut}({}^2B_2(2^e))$ with e is an odd prime.
\end{corollary}
\begin{proof}
When $L\not\cong\mathrm{PSL}_2(q)$ and when $L\not\cong{}^2B_2(q)$, the strongly connectivity of $\Gamma_3(G)$ follows from the strongly connectivity of $\Gamma_3(L)$, which can be deduced from Propositions~\ref{proposition:PSL},~\ref{proposition:PSUUU},~\ref{proposition:symp},~\ref{proposition:omega},~\ref{proposition:omegaplus} and~\ref{proposition:omegaminus} for the classical groups and from Section~\ref{sec:exceptionalspecial} for exceptional Lie groups. Indeed, for every $g\in G\setminus L$, we have ${\bf C}_L(g)\ne 1$ and hence $g$ has in- and out-neighbours in $L$.

We postpone the proof of the case $L=\mathrm{PSL}_2(q)$ to Section~\ref{almostsimplePSL} and the case $L={}^2B_2(q)$ to Section~\ref{Suzuki}.
\end{proof}

It can be easily proved that $\Gamma(G)$ is strongly connected if and only
if $\Gamma(G/{\bf Z}_\infty(G))$ is strongly connected (see~\cite[Lemma 2.1]{DLN}), so putting together the information from Theorems \ref{prel} and \ref{main} and Corollary~\ref{corcorcor}, we achieve a complete classification of the finite groups whose Engel graph is strongly connected.

\begin{corollary}
Let $G$ be a finite group. Then $\Gamma(G)$ is not strongly connected if and only if one of the following cases occur:
\begin{enumerate}
	\item $G/{\bf Z}_\infty(G)$ is a Frobenius groups;
	\item $G/{\bf Z}_\infty(G)\cong \mathrm{PSL}_2(q)$ with $q\ge 4$ even or with $q\equiv 5\pmod 8$;
	\item $G/{\bf Z}_\infty(G)\cong {}^2B_2(q)$ with $q\ge 8$.
	\item $G/{\bf Z}_\infty(G)\cong \mathrm{Aut}({}^2B_2(2^e))$ with e is an odd prime.
\end{enumerate}
\end{corollary}

\section{Notation and preliminaries}\label{not&prel}
In what follows, we write $x\mapsto y$ to denote the arc $(x,y)$ of $\Gamma(G)$ and we write $x\mapsto_n y$ to denote the arc $(x,y)$ of $\Gamma_n(G)$.

To prove Theorem~\ref{main} we use various ingredients. The first is a relation among the Engel graph, the \textit{\textbf{prime graph}} and the \textit{\textbf{commuting graph}}. Given a finite group $G$, we denote by $\pi(G)$ the set of prime divisors of the order of $G$. (More generally, given a positive integer $n$, we denote by $\pi(n)$ the set of prime divisors of $n$.) Now, the prime graph $\Pi(G)$ of $G$ is the graph having vertex set $\pi(G)$ and where two distinct primes $r$ and $s$ are declared to be adjacent if and only if $G$ contains an element having order divisible by $rs$. The commuting graph is the graph having vertex set $G\setminus{\bf Z}(G)$ and where two distinct elements of $G$ are declared to be adjacent if they commute; in our current terminology, we may denote the commuting graph by $\Gamma_1(G)$. 

Suppose now that ${\bf Z}(G)=1$. Observe that $\Gamma_1(G)$ is a subgraph of the Engel graph $\Gamma_n(G)$, in particular, the connected components of the commuting graph $\Gamma_1(G)$ give useful information on the strongly connected components of $\Gamma(G)$. Now, a key result of Williams~\cite{Williams} (see also~\cite[Theorem~4.4]{mp}) gives a method to control the connected components of $\Gamma_1(G)$ using the connected components of $\Pi(G)$.
\begin{theorem}[Theorem~4.4,~\cite{mp}]\label{theorem44}
Let $G$ be a  non-soluble group with ${\bf Z} (G ) = 1$, let $\Psi$ be a connected component of the commuting graph of $G$ and let $\psi:=\Pi(\Psi)$. Suppose $2\notin \psi$. Then $G$ has an abelian Hall $\psi$-subgroup $H$ which is isolated in the commuting graph of $G$ and $\Psi = H\setminus\{1\}$. In particular, $\psi$ is a connected component of the prime graph of $G$.
 \end{theorem}
 Theorem~\ref{theorem44} describes the disconnected components of the commuting graphs of non-soluble groups with trivial center.
They consist of perhaps more than one connected component containing involutions and then all the
remaining connected components are complete graphs. Moreover, these remaining connected components determine (and are determined) by the connected components of the prime graph consisting of odd primes. In particular, the classification of Williams~\cite{Williams} of the connected components of the prime graph of simple groups is important in our work.

 In the case that $G$ has at least two conjugacy classes of involutions the Brauer-Fowler theorem comes to rescue to handle even order elements in the commuting graph.
\begin{theorem}[Lemma~3.5,~\cite{mp}]\label{theorem11}
Let $G$ be a group with trivial center and with at least two conjugacy classes of involutions. Then, the commuting graph has a
unique connected component containing all the elements of even order in $G$.
\end{theorem}
In particular, when $G$ has at least two conjugacy classes of involutions, ${\bf Z}(G)=1$ and $\Pi(G)$ is connected, $\Gamma_n(G)$ is also connected, because so is the commuting graph. Actually, for simple groups of Lie type the work of Morgan and Parker~\cite{mp} simplifies further our analysis.
\begin{theorem}[Proposition~8.10,~\cite{mp}]\label{theorem111}
Let $G$ be a simple group of Lie type. Then, except when $G\cong\mathrm{PSL}_2(q)$, $G\cong{}^2B_2(q)$, $G\cong {}^2G_2(q)$, $G\cong {}^2F_4(q)$ and $G\cong\mathrm{PSL}_3(4)$, the commuting graph of $G$ has a connected component containing all elements of even order of $G$.
\end{theorem}

Combining Theorems~\ref{theorem44} and~\ref{theorem111}, we obtain this useful reduction.
\begin{corollary}\label{cor}
Let $n$ be a positive integer, let $G$ be a simple group of Lie type with $G$ not isomorphic to $\mathrm{PSL}_2(q)$, ${}^2B_2(q)$, ${}^2G_2(q)$, ${}^2F_4(q)$ and $\mathrm{PSL}_3(4)$, and  let $\mathcal{I}$ be the connected component of the commuting graph of  $G$ containing all elements of even order (whose existence is guaranteed by Theorem~$\ref{theorem111}$). Then $\Gamma_n(G)$ is strongly connected if and only if, for every connected component $\psi$ of the prime graph of $G$ with $2\notin\psi$ and for every Hall $\psi$-subgroup $H$ of $G$ (whose existence is guaranteed by Theorem~$\ref{theorem44}$), there exists $h\in H\setminus \{1\}$ and $x,y\in\mathcal{I}$ with $x\mapsto_n h$ and $h\mapsto_n y$.
\end{corollary}

From Corollary~\ref{cor}, the proof of Theorem~\ref{main} is reduced 
\begin{itemize}
\item to $G$ isomorphic to  $\mathrm{PSL}_2(q)$, ${}^2B_2(q)$, ${}^2G_2(q)$, ${}^2F_4(q)$ or $\mathrm{PSL}_3(4)$ (which will be dealt with ad-hoc methods), and
\item  to non-abelian simple groups not isomorphic to one of the groups above (which will be dealt with by studying the connected components of the prime graph). 
\end{itemize}

Before concluding this section we make two comments. First, some of the methods developed in this paper for dealing with the second case above turn out to be also useful to  deal with a conjecture of John Thompson (see~\cite[Problem~8.75]{KM}) on the \textit{\textbf{derangememts}} of a primitive permutation group, see for instance Lemma~\ref{3D4lemma1} and~\ref{2G2lemma1}.
\begin{conjecture}Let $G$ be a finite primitive permutation group on a set $\Omega$. Then, for any two distinct elements $\alpha$ and $\beta$ of $\Omega$, there exists a derangement of $G$ (that is, a fixed-point-free element) with $\alpha^g=\beta$.
\end{conjecture}

Second, Abdollahi~\cite{Ab} has also defined an Engel graph associated to a finite group $G$ (essentially the complement of the undirected version of our graph $\Gamma(G)$). Our definition differs from the definition proposed by Abdollahi and was introduced for the first time by Cameron in~\cite[Section 11.1]{cam}.

In this paper, we denote by ${\bf o}(g)$ the order of the group element $g\in G$ and by ${\bf N}_G(H)$ the normalizer of the subgroup $H\le G$.

In our investigation we will use the following results.

\begin{lemma}\label{2comm}Let $G$ be a group and let  $x,y \in G.$ If ${\bf o}(y)=2,$ then $[x,_ny]=[x,y]^{(-2)^{n-1}}.$
\end{lemma}
\begin{proof}
We prove the statement by induction on $n$. $$[x,_{n+1}y]=[[x,_ny],y]=[[x,y]^{(-2)^{n-1}},y]=[y,x]^{(-2)^{n-1}}y[x,y]^{(-2)^{n-1}}y=[y,x]^{(-2)^{n-1}}[y,x]^{(-2)^{n-1}}=[x,y]^{(-2)^{n}}\qedhere
$$
\end{proof}

The next two lemmas are taken from~\cite{DLN}.

\begin{lemma}\label{norm} Let $G$ be a group and let  $x,y \in G.$ If $x \in {\bf N}_G(\langle y \rangle),$ then $[x,_2y]=1.$
\end{lemma}

\begin{lemma}\label{NC}
	Let $y$ be contained in a subgroup $ K$ of a finite group $G$ and suppose that  the following assumptions are satisfied:
	\begin{enumerate}
		\item\label{eq:NC1} ${\bf N}_G(K)=K$;
		\item\label{eq:NC2} $y\in K^g$ if and only if $K^g=K$.
	\end{enumerate}
	Then $x\notin K$ implies $x\not\mapsto y$.
\end{lemma}

\section{Projective special linear groups of Lie rank $1$}\label{sec:lierank1}
In this section, we consider the Engel graph of the projective special linear group $\mathrm{PSL}_2(q)$. When $q$ is even, the Engel graph of $\mathrm{PSL}_2(q)$ is not strongly connected, see~\cite{DLN}. Therefore, for the rest of this section, we only need to focus on the case $q$ odd. We divide our argument depending on whether $q\equiv 1\pmod 4$ or $q\equiv 3\pmod 4$.
\subsection{Case $q\equiv 1\pmod 4$}\label{sec:q=1}
The main scope of this section is to prove the following technical lemma. Its proof (or better, our proof) relies on some weak random properties of Paley graphs~\cite{aaa}.
\begin{lemma}\label{nr1}
Let $q$ be a prime power with $q\equiv 1\pmod 4$ and with $q\ne 9$, let $\mathbb{F}_q$ be the finite field of cardinality $q$, let $i\in\mathbb{F}_q$ with $i^2=-1$ and let
\[
z:=
\begin{pmatrix}
0&1\\
-1&0
\end{pmatrix}\in \mathrm{SL}_2(q),\quad
z':=
\begin{pmatrix}
-i&0\\
0&i
\end{pmatrix}\in \mathrm{SL}_2(q).
\]
Then, there exists a matrix
\begin{equation}\label{eq:eqeq1}
x:=
\begin{pmatrix}
a&b\\
c&d
\end{pmatrix}\in \mathrm{SL}_2(q)
\end{equation}
with no eigenvalues in $\mathbb{F}_q$
such that
\begin{equation}\label{eq:eqeq2}
x^{-1}z^{-1}xz=[x,z]=z'.
\end{equation}
In particular,
\begin{equation*}
[x,z,z]=\begin{pmatrix}
-1&0\\
0&-1
\end{pmatrix}.
\end{equation*} 
\end{lemma}
\begin{proof}
Using the coefficients $a,b,c,d$ of $x$ in~\eqref{eq:eqeq1}, we see with a matrix computation that
$$
[x,z]=\begin{pmatrix}b^2+d^2&-ab-cd\\-ab-cd&a^2+c^2\end{pmatrix}.$$
In particular, if $[x,z]=z'$, then
\begin{align}\label{eq:eqeq212}
\begin{pmatrix}b^2+d^2&-ab-cd\\-ab-cd&a^2+c^2\end{pmatrix}=\begin{pmatrix}-i&0\\0&i\end{pmatrix}.
\end{align}
As $1=\mathrm{det}(x)=ad-bc$, we obtain
$$0=(-ab-cd)\cdot d=-abd-cd^2=-b-b^2c-cd^2=-b-c(b^2+d^2)=-b+ic.$$
Therefore, $c=-ib$. In particular, $c^2=-b^2$. Summing the elements in coordinate $(1,1)$ and $(2,2)$ of the matrices in~\eqref{eq:eqeq212}, we obtain
$$0=-i+i=(b^2+d^2)+(a^2+c^2)=a^2+d^2$$
and hence $d^2=-a^2$, that is, $d=\varepsilon ia$ for some $\varepsilon\in \{-1,1\}$. Now, the element in coordinate $(1,1)$ in~\eqref{eq:eqeq212} gives
$$-i=b^2-a^2.$$
As $1=\det(x)=ad-bc=\varepsilon ia^2+ib^2$, we get $-i=\varepsilon a^2+b^2$. Therefore, 
$$\varepsilon a^2+b^2=-i=b^2-a^2$$
and hence $\varepsilon=-1$.
Summing up, we have
\[
\left\{
\begin{array}{ccc}
c&=&-ib,\\
d&=&-ia,\\
a^2-b^2&=&i.
\end{array}
\right.
\]
The first two conditions uniquely determine $c$ and $d$ as a function of $a$ and $b$. Whereas, the third condition implies
\begin{equation}\label{eq:eqeq3}
a^2-i \hbox{ is a square in }\mathbb{F}_q.
\end{equation}

Now, the characteristic polynomial of $x$ is
$$T^2+a(i-1)T+1.$$
In particular, $x$ has no eigenvalues in $\mathbb{F}_q$ if and only if the discriminant
\begin{equation*}
(a(i-1))^2-4=-2ia^2-4 \hbox{ is not a square in }\mathbb{F}_q.
\end{equation*}
Dividing $-2ia^2-4$ by $-2i$ and recalling $i^2=-1$, we obtain the equivalent condition
\begin{equation}\label{eq:eqeq4444}
\begin{cases}
a^2-2i \hbox{ is not a square in }\mathbb{F}_q,&\textrm{when }2i \textrm{ is a square},\\
a^2-2i \hbox{ is a square in }\mathbb{F}_q,&\textrm{when }2i \textrm{ is not a square}.
\end{cases}
\end{equation}
Now, observe that $i$ is a square in $\mathbb{F}_q$ if and only if the multiplicative group $\mathbb{F}_q^\ast$ contains an element of order $8$, that is, $q\equiv 1\pmod 8$. Using Gauss quadratic reciprocity and using the fact that $q\equiv 1\pmod 4$, we see that $2$ is a square in $\mathbb{F}_q$ if and only if $q\equiv 1\pmod 8$. In particular, when $q\equiv 1\pmod 8$, $i$ and $2$ are both squares in $\mathbb{F}_q$ and hence so is $2i$. Similarly, when $q\equiv 5\pmod 8$, neither $i$ nor $2$ are squares in $\mathbb{F}_q$ and hence their product $2i$ is a square in $\mathbb{F}_q$. Therefore, $2i$ is a square in $\mathbb{F}_q$ for every $q\equiv 1\pmod 4$. Using this observation,~\eqref{eq:eqeq4444} becomes
\begin{equation}\label{eq:eqeq44444}
a^2-2i \hbox{ is not a square in }\mathbb{F}_q.
\end{equation}
Summing up, the existence of $x\in\mathrm{SL}_2(q)$ satisfying~\eqref{eq:eqeq2} is equivalent to the existence of $a\in\mathbb{F}_q$ satisfying~\eqref{eq:eqeq3} and~\eqref{eq:eqeq44444}.

To prove that there exists $a\in\mathbb{F}_q$ satisfying~\eqref{eq:eqeq3} and~\eqref{eq:eqeq44444} we use an auxiliary graph: this is purely aesthetic, but it makes the argument more clear in our opinion.

Let $P_q$ be the  graph having vertex set $\mathbb{F}_q$ and where two distinct vertices $x$ and $y$ are declared to be adjacent in $P_q$ if $x-y$ is a square in $\mathbb{F}_q$. In algebraic graph theory, $P_q$ is called the Paley graph. The fact that $q\equiv 1\pmod 4$ implies that $P_q$ is undirected. 
Recall that $P_q$ is a strongly regular graph with parameters
$$\left(q,\frac{1}{2}(q-1),\frac{1}{4}(q-5),\frac{1}{4}(q-1)\right).$$
This means that 
\begin{itemize}
\item $P_q$ has $q$ vertices, 
\item each vertex of $P_q$ has valency $(q-1)/2$, 
\item any two distinct adjacent vertices of $P_q$ have $(q-5)/4$ common neighbours, 
\item any two distinct non-adjacent vertices have $(q-1)/4$ common neighbours.
\end{itemize} 

Suppose first that $q\equiv 5\pmod 8$ and recall that in this case $i$ is not a square. In particular, $0$ and $i$ are non-adjacent vertices of $P_q$ and hence $i$ and $0$ have $(q-1)/4$ neighbours in common. Since $2i$ is a square, $0$ and $2i$ are adjacent vertices of $P_q$ and hence  $2i$ and $0$ have $(q-5)/4$ neighbours in common.  Since $(q-5)/4<(q-1)/4$, there exists a vertex $\alpha$ with $\alpha$ adjacent to $0$ and $i$ and with $\alpha$ not adjacent to $2i$. In other words, $\alpha=\alpha-0$ is a square, $\alpha-i$ is a square and $\alpha-2i$ is not a square. In particular, $\alpha$ satisfies~\eqref{eq:eqeq3} and~\eqref{eq:eqeq44444}.

Suppose next that $q\equiv 1\pmod 8$ and recall that in this case $i$ and $2i$ are squares. When $q>(1+2\sqrt{2})^2>14.6$, Theorem~4.1 in~\cite{aaa} applied with $k:=1$ yields that there exists a vertex $\alpha$ adjacent to $0$ and $i$, but $\alpha$ not adjacent  to $2i$. Therefore, as above, $\alpha$ and $\alpha-i$ are squares and $\alpha-2i$ is not a square. In particular, when $q>14.6$, $\alpha$ satisfies~\eqref{eq:eqeq3} and~\eqref{eq:eqeq44444}. Since $q\equiv 1\pmod 8$, the only remaining case is $q:=9$ and it can be verified that the matrix $x$ does not exist in this case.
\end{proof}

We observe that in the last part of the proof of Lemma~\ref{nr1} we did not need to divide the proof depending on the residue class of $q$ modulo $8$, because~\cite{aaa} deals with arbitrary Paley graphs. We have included a different argument when $q\equiv 5\pmod 8$ for its simplicity.

\subsection{Case $q\equiv 3\pmod 4$: character theoretic considerations}\label{sec:q=3}
In this section, we are mainly concerned with the case that $q\equiv 3\pmod 4$. However, since some of our arguments are rather general, we make a restriction on $q$ only when needed.

Let $q:=p^f$ be a prime 
power with $p$ an odd prime.
We start with some considerations concerning $\mathrm{SL}_2(q)$. We tend to use a ``tilde'' to denote the elements of the special linear group $\mathrm{SL}_2(q)$.
We consider the action of $\mathrm{SL}_2(q)$ on the projective line $\Omega$ and on the non-zero row vectors of $\mathbb{F}_q^2\setminus\{0\}$. Fix $e_1:=(1,0)\in \mathbb{F}_q^2$, let $\omega:=\langle (1,0)\rangle=\langle e_1\rangle\in \Omega$, let $$\tilde{P}:=
\left\{\begin{pmatrix}1&0\\ x&1\end{pmatrix}\mid x\in\mathbb{F}_q\right\}$$ be the stabilizer of $e_1$ in $\mathrm{SL}_2(q)$. In particular, $\tilde{P}$ is a Sylow $p$-subgroup of $\mathrm{SL}_2(q)$. Let 
\[
\tilde{z}:=\begin{pmatrix}0&1\\-1&0\\
\end{pmatrix}.
\]
Finally, when $q\equiv 1\pmod 4$, let $i\in\mathbb{F}_q$ with $i^2=-1$.
\begin{lemma}\label{le:2}
We have
\begin{equation}\label{eq:55}
\{e_1\cdot {\tilde{g}^{-1}\tilde{z}\tilde{g}}\mid \tilde{g}\in \mathrm{SL}_2(q)\}=
\begin{cases}
\mathbb{F}_q^2\setminus\omega&\textrm{when }q\equiv 3\pmod 4,\\
(\mathbb{F}_q^2\setminus\omega)\cup\{ie_1,-ie_1\}&\textrm{when }q\equiv 1\pmod 4.
\end{cases}
\end{equation}
\end{lemma}
\begin{proof}
Let $\omega':=\omega^{\tilde{z}}$, let $e_2:=(0,1)\in \omega'$ and observe that $e_1\tilde{z}=e_2$. 
For each $\lambda\in\mathbb{F}_q\setminus\{0\}$, let $a,b\in \mathbb{F}_q$ satisfying $\lambda^2a^2+b^2=\lambda$. The existence of $a$ and $b$ is guaranteed by an elementary counting argument; indeed, the set $\{a^2\mid a\in\mathbb{F}_q\}$ has cardinality $(q+1)/2>|\mathbb{F}_q|/2$. Now, consider
\[
\tilde{g}:=\begin{pmatrix}
a&b\\
-\lambda^{-1} b&\lambda a
\end{pmatrix}
\]
and observe that $\det(\tilde{g})=\lambda a^2+\lambda^{-1}b^2=\lambda^{-1}(\lambda^2 a^2+b^2)=\lambda^{-1}\lambda=1$ and hence $\tilde{g}\in \mathrm{SL}_2(q)$. We have
\[
\tilde{g}^{-1}\tilde{z}\tilde{g}=\begin{pmatrix}
0&\lambda\\
-\lambda^{-1}&0
\end{pmatrix}
\]
and hence $e_1\tilde{g}^{-1}\tilde{z}\tilde{g}=(0,\lambda)=\lambda e_2$. Since $\lambda$ is an arbitrary element of $\mathbb{F}_q\setminus\{0\}$, this shows that
$$\{\lambda e_2\mid \lambda\in\mathbb{F}_q\setminus\{0\}\}\subseteq \{e_1{\tilde{g}^{-1}\tilde{z}\tilde{g}}\mid \tilde{g}\in \mathrm{SL}_2(q)\}.$$

Observe that
$$
\bigcup_{\tilde{u}\in\tilde{P}}
\{\lambda e_2\mid \lambda\in\mathbb{F}_q\setminus\{0\}\}
\cdot \tilde{u}=
\bigcup_{\tilde{u}\in\tilde{P}}
\{\lambda e_2\tilde{u}\mid \lambda\in\mathbb{F}_q\setminus\{0\}\}
=\mathbb{F}_q^2\setminus\langle e_1\rangle=\mathbb{F}_q^2\setminus\omega.$$
Moreover, for each $\tilde{g}\in \mathrm{SL}_2(q)$ and for each $\tilde{u}\in \tilde{P}$, as $e_1\tilde{u}=e_1$, we have
$$e_1\tilde{g}^{-1}\tilde{z}\tilde{g}\tilde{u}=e_1\tilde{u}^{-1}\tilde{g}^{-1}\tilde{z}\tilde{g}\tilde{u}=e_1(\tilde{g}\tilde{u})^{-1}\tilde{z}(\tilde{g}\tilde{u}).$$ Therefore,
$$\{e_1{\tilde{g}^{-1}\tilde{z}\tilde{g}}\mid \tilde{g}\in \mathrm{SL}_2(q)\}\supseteq\mathbb{F}_q^2\setminus\langle e_1\rangle=\mathbb{F}_q^2\setminus\omega.$$

When $q\equiv 3\pmod 4$, the matrix $\tilde{z}$ has no fixed points in its action on the projective line $\Omega$ and hence there exists no $\tilde{g}\in \mathrm{SL}_2(q)$ with $\omega{\tilde{g}^{-1}\tilde{z}\tilde{g}}=\omega$. This proves~\eqref{eq:55} when $q\equiv 3\pmod 4$.

When $q\equiv 1\pmod 4$, it only remains to determine the vectors in $\omega=\langle e_1\rangle$ of the form $e_1\tilde{g}^{-1}\tilde{z}\tilde{g}$, for some $\tilde{g}\in\mathrm{SL}_2(q)$. In particular, $\tilde{g}^{-1}\tilde{z}\tilde{g}$ fixes the projective point $\omega$. An elementary matrix computation shows that, if $\tilde{g}^{-1}\tilde{z}\tilde{g}$ fixes $\omega$, then
\[
\tilde{g}^{-1}\tilde{z}\tilde{g}=\pm\begin{pmatrix}

i&0\\
0&-i
\end{pmatrix}.
\] 
This proves~\eqref{eq:55} when $q\equiv 1\pmod 4$.
\end{proof}

We now work in the projective special linear group $\mathrm{PSL}_2(q)$, we omit the tilde to denote the corresponding elements and subgroups of $G:=\mathrm{PSL}_2(q)$. For instance, $P$ is the projection of $\tilde{P}$ in $\mathrm{PSL}_2(q)$. Now let $x\in\mathrm{PSL}_2(q)$ be an element of order $p$.  Replacing $x$ by a suitable conjugate, we may suppose that $x\in P$. Observe that ${\bf C}_{\mathrm{PSL}_2(q)}(x)=P$ and that $\mathrm{PSL}_2(q)$ has two conjugacy classes of elements of order $p$ with representatives $x_+$ and $x_-$, say. 

Set $$\mathcal{C}_+:=x_+^G=\{g^{-1}\cdot x_+\cdot g\mid g \in G\}\,\hbox{ and }\,\mathcal{C}_-:=x_-^G=\{g^{-1}\cdot x_-\cdot g\mid g \in G\}.$$ Now, let $z$ be an involution in $G$ and observe that $G$ has a unique conjugacy class of involutions. Let 
\begin{align*}
\mathcal{C}&:=z^G,\\
\mathfrak{C}_+&:=\{x_+^g\mid g\in \mathcal{C}\}=\{g^{-1}\cdot x_+\cdot g\mid g\in \mathcal{C}\}\subseteq\mathcal{C}_+,\\
\mathfrak{C}_-&:=\{x_-^g\mid g\in \mathcal{C}\}=\{g^{-1}\cdot x_-\cdot g\mid g\in\mathcal{C}\}\subseteq\mathcal{C}_-.\end{align*}
When $q\equiv 1\pmod 4$, let 
$$\tilde{\iota}:=\begin{pmatrix}i&0\\0&i^{-1}\end{pmatrix}.$$
\begin{lemma}\label{lemma21}
We have
\begin{equation}\label{eq:1}
P\mathcal{C} = 
\begin{cases}
G\setminus {\bf N}_G(P),&\textrm{when }q\equiv 3\pmod 4,\\
(G\setminus {\bf N}_G(P))\cup P\iota\,&\textrm{when }q\equiv 1\pmod 4,\\
\end{cases}
\end{equation}
where $P\mathcal{C}:=\{gh\mid g\in P,h\in \mathcal{C}\}$ and $\iota$ is the projection of $\tilde{\iota}$ in $\mathrm{PSL}_2(q)$. 
\end{lemma}
\begin{proof}
Observe that $\tilde{g}\in \mathrm{SL}_2(q)$ fixes $\omega=\langle e_1\rangle$ if and only if $\tilde{g}\in{\bf N}_{\mathrm{SL}_2(q)}{(\tilde{P})}$. Therefore 
$$\{\tilde{g}\in\mathrm{SL}_2(q)\mid e_1\tilde{g}\in \mathbb{F}_q^2\setminus \omega\}=\mathrm{SL}_2(q)\setminus {\bf N}_{\mathrm{SL}_2(q)}{(\tilde{P})}.$$
For the time being, assume $q\equiv 3\pmod 4$.
Now, let $\tilde{x}\in\tilde{P}$ and let $\tilde{g}\in \mathrm{SL}_2(q)$. Then, since $e_1\tilde{x}=e_1$, from~\eqref{eq:55}, we get 
$e_1\tilde{x}\tilde{g}^{-1}\tilde{z}\tilde{g}\in\mathbb{F}_q^2\setminus \omega$. Hence $\tilde{x}\tilde{g}^{-1}\tilde{z}\tilde{g}\in \mathrm{SL}_2(q)\setminus {\bf N}_{\mathrm{SL}_2(q)}{(\tilde{P})}.$ Therefore,
\begin{equation}\label{diavolo2}\tilde{P}\{\tilde{g}^{-1}\tilde{z}\tilde{g}\mid \tilde{g}\in \mathrm{SL}_2(q)\}\subseteq \mathrm{SL}_2(q)\setminus {\bf N}_{\mathrm{SL}_2(q)}{(\tilde{P})}.\end{equation}
 Conversely, given $\tilde{y} \in\mathrm{SL}_2(q)\setminus {\bf N}_{\mathrm{SL}_2(q)}{(\tilde{P})}$, we have $e_1\tilde{y}\in \mathbb{F}_q^2\setminus \omega$ and hence, from~\eqref{eq:55}, there exists $\tilde{g}\in \mathrm{SL}_2(q)$ with $e_1\tilde{y}=e_1\tilde{g}^{-1}\tilde{z}\tilde{g}$. Therefore, $\tilde{y}(\tilde{g}^{-1}\tilde{z}\tilde{g})^{-1}$ fixes the vector $e_1$ and hence $\tilde{y}(\tilde{g}^{-1}\tilde{z}\tilde{g})^{-1}=\tilde{x}\in\tilde{P}$. Thus
$\tilde{y}=\tilde{x}\tilde{g}^{-1}\tilde{z}\tilde{g}$. Therefore,
\begin{equation}\label{diavolo1}\mathrm{SL}_2(q)\setminus {\bf N}_{\mathrm{SL}_2(q)}{(\tilde{P})}\subseteq \tilde{P}\{\tilde{g}^{-1}\tilde{z}\tilde{g}\mid \tilde{g}\in \mathrm{SL}_2(q)\}.
\end{equation}
Projecting~\eqref{diavolo2} and~\eqref{diavolo1} in $\mathrm{PSL}_2(q)$, we obtain~\eqref{eq:1}.

The argument when $q\equiv 1\pmod 4$ is entirely similar.
\end{proof}

\begin{lemma}
Let $\varepsilon\in \{-,+\}$. We have
\begin{equation}\label{eq:2}
\begin{cases}
\mathfrak{C}_\varepsilon=\mathcal{C}_\varepsilon\setminus (\mathcal{C}_\varepsilon\cap P),&\textrm{when }q\equiv 3\pmod 4,\\
\mathfrak{C}_\varepsilon=(\mathcal{C}_\varepsilon\setminus (\mathcal{C}_\varepsilon\cap P))\cup\{x_\varepsilon^{-1}\},&\textrm{when }q\equiv 1\pmod 4.
\end{cases}
\end{equation}
\end{lemma}
\begin{proof}
Observe that $P={\bf C}_G(x_\varepsilon)$. In particular, $\mathfrak{C}_\varepsilon=\{x_\varepsilon^{gc}\mid g\in P,c\in \mathcal{C}\}$. Now, the set $P\mathcal{C}$ is described in~\eqref{eq:1}. Observe that, $\{x_\varepsilon^{x}:x\in {\bf N}_G(P)\}=\mathcal{C}_\varepsilon\cap P$. In particular,~\eqref{eq:2} follows immediately when $q\equiv 3\pmod 4$. When $q\equiv 1\pmod 4$, we also need to take into account the contribution given by $P\iota$ in $P\mathcal{C}$. A computation gives $x_\varepsilon^\iota=x_\varepsilon^{-1}$ and hence ~\eqref{eq:2} follows also when $q\equiv 1\pmod 4$.
\end{proof}

For the next lemma we need to observe that, for each $\varepsilon\in \{+,-\}$, $x_\varepsilon,x_\varepsilon^{-1}$ are in the same $G$-conjugacy class when $q\equiv 1\pmod 4$, whereas they are in distinct conjugacy classes when $q\equiv 3\pmod 4$. In particular, when $q\equiv 3\pmod 4$, $x_\varepsilon^{-1}\in \mathcal{C}_{-\varepsilon}.$

\begin{lemma}\label{lemma:3}
Assume $q\equiv 3\pmod 4$. For every $\varepsilon\in \{+,-\}$, there exists $g\in z^G=\mathcal{C}$ with
$[x_\varepsilon,g]$ having order a power of $2$. Moreover, if ${\bf o}([x_\varepsilon,g])=2^a$ for some $a\ge 1$, then $2^a$ is the largest power of $2$ dividing $(q+1)/2$.

Assume $q\equiv 1\pmod 4$. For every $\varepsilon\in \{+,-\}$, there exists $g\in z^G=\mathcal{C}$ with
$[x_\varepsilon,g]$ having order a power of $2$ if and only if $q\not\equiv 5\pmod 8$. 
\end{lemma}
\begin{proof}
Let $\mathcal{C}_1,\ldots,\mathcal{C}_\kappa$ be the conjugacy classes of $G$ and let $x_1,\ldots,x_\kappa$ be a set of representatives with $x_i\in \mathcal{C}_i$, for every $i\in \{1,\ldots,\kappa\}$. For any conjugacy class $\mathcal{C}_i$ in $G$, we define 
$$\hat{\mathcal{C}}_i:=\sum_{c\in \mathcal{C}_i}c\in\mathbb{C}G,$$
where $\mathbb{C}G$ is the complex group algebra over $G$. Then
$$\hat{\mathcal{C}}_i\hat{\mathcal{C}}_j=\sum_{v=1}^ka_{ijv}\hat{\mathcal{C}}_v,$$
where $a_{ijv}\in\mathbb{N}$ are the class constants of $G$. There is a combinatorial interpretation of the $a_{ijv}$, which comes from the conjugacy class association scheme of $G$. Indeed, 
\begin{equation}\label{scheme}a_{ijv}:=|\{(a,b)\in \mathcal{C}_i\times\mathcal{C}_j\mid ab=x_v\}|=|\{(b,c)\in\mathcal{C}_j\times\mathcal{C}_v\mid x_ib=c\}|.
\end{equation} The class constants can be computed using the character table of $G=\mathrm{PSL}_2(q).$ From~\cite[(3.9)]{Isaacs}, we have
$$\frac{|G|}{|\mathcal{C}_i||\mathcal{C}_j|}a_{ijv}:=\sum_{\chi\in \mathrm{Irr}(G)}\frac{\chi(x_i)\chi(x_j)\chi(x_v^{-1})}{\chi(1)}.$$

Let $g\in z^G$. Now observe that, using the notation that we have established above and~\eqref{eq:2}, we have
$$
[x_\varepsilon,g]=x_\varepsilon^{-1}x_\varepsilon^g\in
 x_\varepsilon^{-1}\mathfrak{C}_{\varepsilon}=
 \begin{cases}
x_\varepsilon^{-1}\mathcal{C}_\varepsilon\setminus (x_\varepsilon^{-1}(\mathcal{C}_\varepsilon\cap P)) &\textrm{when }q\equiv 3\pmod 4,\\
x_\varepsilon^{-1}\mathcal{C}_\varepsilon\setminus (x_\varepsilon^{-1}(\mathcal{C}_\varepsilon\cap P))\cup\{x_\varepsilon^{-2}\} &\textrm{when }q\equiv 1\pmod 4.
 \end{cases}$$
 This shows that $x_\varepsilon^{-1}\mathfrak{C}_{\varepsilon}$  differs from $x_\varepsilon^{-1}\mathcal{C}_\varepsilon$ only by elements having order $p$. In particular, we may investigate the elements having order a power of $2$ in $x_{\varepsilon}^{-1}\mathfrak{C}_{\varepsilon}$, by investigating the elements having order a power of $2$ in $x_\varepsilon^{-1}\mathcal{C}_\varepsilon$ using the class constants and the character formula above.

Now, the character table of $\mathrm{PSL}_2(q)$ is well-known and can be found online. A good reference for the character table of $\mathrm{PSL}_2(q)$ is also~\cite{Fritzsche}. Its structure depends on the congruence of $q$ modulo $4$.

Assume $q\equiv 3\pmod 4$. Let $y\in G$ having order $(q+1)/2$ and let $\ell$ be a divisor of $(q+1)/2$ with $\ell<(q+1)/2$. Let $\mathcal{C}_\ell:=(y^\ell)^G$. We may use the character table of $G$ for computing $a_{ijv}$, where $\mathcal{C}_i=\mathcal{C}_{-\varepsilon}$, $\mathcal{C}_j=\mathcal{C}_\varepsilon$ and $\mathcal{C}_v=\mathcal{C}_\ell$, we find
$$
\sum_{\chi\in\mathrm{Irr}(G)}\frac{\chi(x_\varepsilon)\chi(x_\varepsilon^{-1})\chi(y^{-\ell})}{\chi(1)}=
\begin{cases}
0&\textrm{ when }\ell \textrm{ is even},\\
\frac{2q}{q-1}&\textrm{ when }\ell \textrm{ is odd}.\\
\end{cases}
$$
Now the lemma immediately follows. Write $(q+1)/2:=2^a b$, with $b$ odd. Apply the argument above with $\ell:=b$, the element $y^b$ has order $2^a$.

Assume $q\equiv 1\pmod 4$. In this case any 2-element of $G$ is a power of a suitable element of order $(q-1)/2$. Let $y\in G$ having order $(q-1)/2$ and let $\ell$ be a divisor of $(q-1)/2$ with $\ell<(q-1)/2$. Let $\mathcal{C}_\ell:=(y^\ell)^G$. We may use the character table of $G$ for computing $a_{ijv}$, where $\mathcal{C}_i=\mathcal{C}_{\varepsilon}$, $\mathcal{C}_j=\mathcal{C}_\varepsilon$ and $\mathcal{C}_v=\mathcal{C}_\ell$, we find
$$
\sum_{\chi\in\mathrm{Irr}(G)}\frac{\chi(x_\varepsilon)^2\chi(y^{-\ell})}{\chi(1)}=
\begin{cases}
0&\textrm{ when }\ell \textrm{ is odd},\\
\frac{2q}{q+1}&\textrm{ when }\ell \textrm{ is even}.\\
\end{cases}
$$
Now the lemma immediately follows. Write $(q-1)/2:=2^a b$, with $b$ odd. The element $y^\ell$ has order a power of $2$ only when $\ell$ is divisible by $b$. However, the above computation shows that $y^\ell$ is conjugate to an element of the form $[x_\varepsilon,g]$, for some $g\in z^G$, only when $\ell$ is even. Therefore $2b$ divides $\ell$. When $a=1$, we have $\ell=(q-1)/2$ and hence $y^\ell$ is the identity element and there is no $2$-element of the form  $[x_\varepsilon,g]$. This latter case occurs when $(q-1)/2$ is twice an  odd number, that is, $q\equiv 5\pmod 8$.
\end{proof}

\subsection{The Engel graph of $\mathrm{PSL}_2(q)$}\label{sec:lie1}
Let $n$ be a positive integer. Our aim is to study the strong connectivity of the graphs $\Gamma_n(\mathrm{PSL}_2(q))$ for any prime power $q.$ As we mentioned at the start of Section~\ref{sec:lierank1}, when $q$ is even, it is already known that there is no $n$ for which $\Gamma_n(\mathrm{PSL}_2(q))$ is strongly connected, so we will assume that $q$ is odd.

\begin{lemma}\label{3comm}
Let $G=\mathrm{PSL}_2(q)$ with $q=p^f$, where $p$ is an odd prime. If ${\bf o}(x)=p$, then
$x\mapsto_n y$ if and only if 
\begin{itemize}
\item either $y \in \langle x\rangle$, or
\item ${\bf o}(y)=2$ and $[x,y]^{2^{n-1}}=1.$
\end{itemize}
\end{lemma}

\begin{proof}
We distinguish the different possibilities for the order of $y.$

 Assume ${\bf o}(y)=p$ and let $K$ the unique parabolic subgroup of $G$ containing $y.$ By Lemma~\ref{NC}, $x\mapsto y$ if and only if $y\in K$, i.e., if and only if $y \in {\bf C}_G(x)=\langle x \rangle$.

 Assume ${\bf o}(y)$ divides $(q+\varepsilon)/2,$ 
	with $\varepsilon=\pm 1$ and let $K={\bf N}_G(\langle y\rangle)$. Then $K$ is a dihedral group of order $q+\varepsilon$ and,  by Lemma \ref{NC}, $x\mapsto y$ only if $y$ is an involution. In the latter case, by Lemma~\ref{2comm}, $x\mapsto_n y$ if and only if $[x,y]^{2^{n-1}}=1.$
\end{proof}

\begin{lemma}\label{cruciale}
	Let $G=\mathrm{PSL}_2(q)$  with $q=p^f$, where $p$ is an odd prime, and let $n\geq 2.$ Then $\Gamma_n(G)$ is strongly connected if and only if, for any $x\in G$ of order $p$, there exists $y$ in $G$ with  ${\bf o}(y)=2$ and $[x,y]^{2^{n-1}}=1.$
\end{lemma}

\begin{proof} 
By Lemma \ref{2comm}, if for some $x$ of order $p,$ there is no involution $y$ with the property that the order of $[x,y]$ divides $2^{n-1}$, then there is no path in the directed graph $\Gamma_n(G)$ starting from $x$ and ending to an element $z \notin \langle x \rangle,$ so $\Gamma_n(G)$ is not strongly connected. 
	
Conversely assume that, for any  $x\in G$ of order $p$, there exists $y$ in $G$ with  ${\bf o}(y)=2$ and $[x,y]^{2^{n-1}}=1.$	
Let $\varepsilon \in \{-1,1\}$ so that
4 divides $q-\varepsilon.$ Recall that, if $g\in G$, then either ${\bf o}(g)=p$ or ${\bf o}(g)$ divides $(q\pm 1)/2.$
Denote by $A$ the set of the elements of $G$ of order $p$ and by $B_\varepsilon$ and $B_{-\varepsilon}$ the set of non-identity elements of $G$ whose order divides, respectively, $(q+\varepsilon)/2$ and $(q-\varepsilon)/2.$
 
By~\cite[Theorem 1.1]{bbpr}, all the involutions of $G$, and consequently all the elements whose centralizer has even order,  belong to the same connected component of the commuting graph $\Gamma_1(G).$ Let $\Omega$ be the strong component of $\Gamma_n(G)$ which contains the elements of even order. We have $B_{-\varepsilon} \subseteq \Omega.$ For $X, Y \in \{A,B_{-\varepsilon}, B_{\varepsilon}\},$ we write $X \to_n Y$ to denote that there exists $a \in A$ and $b \in B$ with $a \mapsto_n b.$ By assumption and Lemma \ref{2comm}, for any $a \in A$, there exists an involution $y$ with $a\mapsto_n y$, so $A \to_n B_{-\varepsilon}.$

 First assume $q\equiv 3\pmod 4$. 
 If $g\in B_\varepsilon$, then there is an involution $y$ in ${\bf N}_G(\langle g \rangle)$  and $y \mapsto_2 g$ by Lemma \ref{norm}. 
 Moreover $g$ normalizes $\langle x\rangle$ for some $x\in A,$ so, again by Lemma~\ref{norm},
 $g \mapsto_2 x$. Hence  $B_{-\varepsilon} \to_n B_{\varepsilon}\to_n A\to B_{-\varepsilon}$ and consequently $A\cup B_\varepsilon \cup B_{-\varepsilon} \subseteq \Omega.$

Now assume $q\equiv 1\pmod 4$. If $g \in A$, then $\langle g \rangle$ is normalized by an element of $B_{-\varepsilon}$ so $A \to_n B_{-\varepsilon}\to A.$ Moreover $B_\varepsilon \to_n B_{-\varepsilon},$ by Lemmas~\ref{2comm} and~\ref{nr1}, and
$B_{-\varepsilon} \to_n B_{\varepsilon},$ because the normalizer of the subgroup generated by an element of $B_{\varepsilon}$ contains an involution.
\end{proof}

\begin{theorem}\label{thrm:tired}Let $G=\mathrm{PSL}_2(q)$  with $q$ odd and let $n\geq 2.$
	\begin{itemize}
		\item If $q\equiv 3\pmod 4$, then $\Gamma_n(G)$ is strongly connected if and only if $n > a$ where $2^a$ is the largest power of $2$ dividing $(q+1)/2$.
		\item If $q\equiv 1\pmod 8$ and $q\ne 9$, then  $\Gamma_2(G)$ is strongly connected.
\item If $q=9$, then $\Gamma_n(G)$ is strongly connected if and only if $n\geq 3.$
\item If $q\equiv 5\pmod 8$, then $\Gamma(G)$ is not strongly connected.
	\end{itemize}
\end{theorem}
\begin{proof}When $q\ne 9$, the proof follows immediately combining Lemmas~\ref{nr1},\ref{lemma:3} and~\ref{cruciale}. When $q=9$, proof follows with a computation.  
\end{proof}

\begin{corollary}
The Engel graph $\Gamma(\mathrm{PSL}_n(q))$ is strongly connected if and only if  $q\not\equiv 5\pmod 8$ or  $q$ is even.
\end{corollary}

\subsection{The Engel graph of $\mathrm{PGL}_2(q)$}

Assume $G=\mathrm{PGL}_2(q)$ with $q=p^f$, where $p$ is an odd prime. Then $G$ has two conjugacy classes of involutions, and by \cite[Lemma 3.5]{mp},
there is a unique connected component, say $\Omega,$ of $\Gamma_1(G)$ containing all the elements in $G$ whose centralizer has even order.

Then $\Omega$ contains all the elements of $G$, except the ones of order $p.$  Recall that, if $g$ has order $p$ and $P$ is a Sylow $p$-subgroup of $G$ with $g\in  P$, then ${\bf N}_G(P)$ contains an element $z$ of order $q-1,$ and consequently $z \mapsto_2 g.$ 
Conversely, up to conjugation, we may assume $g=\begin{pmatrix}1&a\\0&1\end{pmatrix}Z$, where $Z:={\bf Z}(\mathrm{GL}_2(q))$ is the center of $\mathrm{GL}_2(q)$.
Consider $y=\begin{pmatrix}0&b\\1&0\end{pmatrix}Z.$
We have 
$$u:=\begin{aligned}
	\begin{pmatrix}1&a\\0&1\end{pmatrix} \begin{pmatrix}0&b\\1&0\end{pmatrix}
	\begin{pmatrix}1&-a\\0&1\end{pmatrix} \begin{pmatrix}0&b\\1&0\end{pmatrix}=
	\begin{pmatrix}-a^2+b&ab\\-a&b\end{pmatrix}
\end{aligned}
$$
Choose $b=a^2/2.$
Then 
$$u^2=\begin{pmatrix}a^4+b^2-3a^2b&ab(2b-a^2)\\a(a^2-2b)&-a^2b+b^2\end{pmatrix}=
\begin{pmatrix}-b^2&0\\0&-b^2	\end{pmatrix}
.$$
Hence $[g,y]^2=1$ and $g\mapsto_2 y.$ 
We have so proved:
\begin{proposition}\label{andrea}If $q$ is odd, then $\Gamma_2(\mathrm{PGL}_2(q))$ is strongly connected.
\end{proposition}
When $q$ is even, $\mathrm{PSL}_2(q)=\mathrm{PGL}_2(q)$ is not strongly connected, see~\cite{DLN}.
\subsection{Almost simple groups having socle $\mathrm{PSL}_2(q)$}\label{almostsimplePSL}
The main scope in this section is determining the strongly connectivity of $\Gamma_n(G)$, where $G$ is an almost simple group with socle $L:=\mathrm{PSL}_2(q)$. 

\begin{theorem}\label{thrm:tiredtired}
Let $G$ be an almost simple group with socle $L:=\mathrm{PSL}_2(q)$ and with $L<G$. Then $\Gamma_2(G)$ is strongly connected.
\end{theorem}
\begin{proof}
In the light of Proposition~\ref{andrea}, we may suppose $L<G\ne \mathrm{PGL}_2(q)$.

Let $e:=|G:L|$ and suppose first that $e$ is even. Then $G\ge H\ge L$ with $|H:L|=2$. Now, we have at most three choices for $H$: $H=\mathrm{PSL}_2(q)\langle\iota\rangle=\mathrm{PGL}_2(q)$, $H=\mathrm{PSL}_2(q)\langle\alpha\rangle=\mathrm{P}\Sigma\mathrm{L}_2(q)$ or $H=\mathrm{PSL}_2(q)\langle\iota\alpha\rangle$, where $\alpha:\mathbb{F}_q\to\mathbb{F}_q$ is an involutory field automorphism and $\iota$ is the projective image of
\[
\begin{pmatrix}
0&1\\
1&
0
\end{pmatrix}.
\]

In the first case, $G\ge\mathrm{PGL}_2(q)$ and $q$ is odd and, for this case, the strongly connectivity of $\Gamma_2(G)$ follows using the strongly connectivity of $\Gamma_2(\mathrm{PGL}_2(q))$ in Proposition~\ref{andrea}.

In the second case and in the third case, $q=q_0^2$ for some prime power $q_0$. If $q_0$ is odd, then Theorem~\ref{thrm:tired} gives that $\Gamma_2(\mathrm{PSL}_2(q))$ is strongly connected, because $q=q_0^2\equiv 1\pmod 8$. (The exceptional case of $q=9$ can be handled with a computer.) Assume then $q_0$ even. Observe that $\iota\in\mathrm{SL}_2(q)$ and hence we only need to deal with the second case. Now, $G$ has two conjugacy classes of involutions, and by~\cite[Lemma 3.5]{mp}, there is a unique connected component, say $\Omega,$ of $\Gamma_1(G)$ containing all the elements of even order in $G.$ Since the centralizer in $L$ of $\alpha$  contains a subgroup isomorphic to $\mathrm{SL}_2(q_0)$,
$\Omega$ contains all the elements of $L=\mathrm{PSL}_2(q)$ of order $2$ of or order a divisor of $q-1$. Therefore, it sufficies to consider the elements having order a divisor of $q+1$. Let 
\[
g=
\begin{pmatrix}
0&1\\
1&x
\end{pmatrix}\in\mathrm{SL}_2(q),
\]
having order $q+1$ for some $x\in\mathbb{F}_q$. Then
$$[g,\alpha]=
\begin{pmatrix}
1&x+x^\alpha\\
0&1
\end{pmatrix}.
$$
As $[g,\alpha]$ has order $2$, we deduce $g\mapsto_2\alpha$ and from this it follows immediately that $\Gamma_2(G)$ is strongly connected. 

\smallskip 

Finally, we suppose that $|G:L|=e$ is odd. In particular, $G=L\rtimes\langle\alpha\rangle$ for some field automorphism of odd order $e$. Set $q_0:=q^{1/e}$ and observe that, since $e$ is odd, $q_0-1$ divides $q-1$ and $q_0+1$ divides $q+1$. Now, the centralizer of $\alpha$ in $L$ contains a subgroup isomorphic to $\mathrm{PSL}_2(q_0)$.
We claim that any two elements of $L$ having order $2$ are in the same strongly connected component of $L$. Let $x,y\in L$ with ${\bf o}(x)={\bf o}(y)=2$. Let $P$ be a Sylow $2$-subgroup of $L$ with $x\in P$. If $y\in P$, then $[x,y,y]=1$ because $P$ has nilpotency class $2$. Therefore, suppose $y\notin P$. Let $P^-$ be the opposite Sylow $2$-subgroup of $L$ (here we are thinking of fixing a root system for the Lie group $\mathrm{PSL}_2(q)$). Now, $P$ acts transitively on the set of Sylow $2$-subgroups of $L$ distinct from $P$. Therefore, replacing $y$ and $x$ with a suitable $P$-conjugate, we may suppose that $y\in P^-$. 
Replacing the field authomorphism $\beta$ with a suitable $L$-conjugate if necessary, we may suppose that $Q:=P\cap {\bf C}_L(\alpha)$ is a Sylow $2$-subgroup of ${\bf C}_L(\alpha)=\mathrm{PSL}_2(q_0)$. In particular, $Q^-=P^-\cap {\bf C}_L(\alpha)$ is the opposite Sylow $2$-subgroup of ${\bf C}_L(\alpha)=\mathrm{PSL}_2(q_0)$. Now, let $x_0\in Q\setminus \{1\}$ and $y_0\in Q^-\setminus\{1\}$. We have $[x,x_0,x_0]=1$, $[x_0,\beta]=1$, $[\beta,y_0]=1$ and $[y_0,y,y]=1$, that is,
$$x\mapsto_2 x_0\mapsto_1\alpha\mapsto_1 y_0\mapsto_2y,$$ 
which is what we wanted to prove. 

Let $\Omega$ be a strongly connected component of $\Gamma_2(G)$ containing an element of order $2$. From the paragraph above, $\Omega$ contains all the elements having even order and contains all the $L$-conjugates of $\alpha$.

From these facts, it is not hard to deduce  that $\Gamma_2(G)$ is strongly connected, except when $q_0-1=1$. Indeed, when $q_0-1=1$, we cannot guarantee that we reach the elements of order $q-1$ in $L$ from elements of ${\bf C}_L(\alpha)$. 

Suppose $q_0-1=1$, that is, $q_0=2$. Let $a\in\mathbb{F}_q^\ast$ having order $q-1$ and set
\[
x:=
\begin{pmatrix}
a^{-1}&0\\
0&a\\
\end{pmatrix},
z:=
\begin{pmatrix}
1&0\\
1&1\\
\end{pmatrix},
z':=
\begin{pmatrix}
0&1\\
1&0\\
\end{pmatrix}.
\]
Then
\[
[x,z]=\begin{pmatrix}
1&0\\
a^{-2}+1&1\\
\end{pmatrix}
\]
has order $2$ and hence $x\mapsto_2 z$. Similarly, $[z',x]=x^{-1}$; thus $[z',x,x]=1$ and $z'\mapsto_2 x$.
  Thus $\Gamma_2(G)$ is strongly connected also in this case.

 From these comments it easily follows that $\Gamma_2(G)$ is connected.
\end{proof}
\section{Projective special linear groups of Lie rank $>1$}\label{sec:lie>1}
Before embarking in studying the connectivity of the Engel graph of $\mathrm{PSL}_m(q)$ with $m\ge 3$, we prove a technical lemma.
\begin{lemma}\label{PSL}
Let $m$ be an odd integer with $m\ge 3$, let $q$ be a prime power and let $\mathbb{F}_q$ be a finite field of cardinality $q$. Then there exists an element $g\in\mathrm{SL}_m(q)$ acting irreducibly on $\mathbb{F}_q^m$ and an involution $z\in \mathrm{SL}_m(q)\setminus {\bf Z}(\mathrm{SL}_m(q))$ such that 
$${\bf o}([g,z])=
\begin{cases}
2,&\textrm{when }q\textrm{ is odd, }\\
4,&\textrm{when }q\textrm{ is even.}
\end{cases}
$$
\end{lemma}
\begin{proof}
 Using Hilbert's theorem 90, we may deduce that there exists an irreducible polynomial 
$$p(x)=x^m-a_1x^{m-1}-\cdots-a_{m-1}x-1$$ over $\mathbb{F}_q$. Now consider the companion matrix of the polynomial $p(x)$:
\[
g:=\begin{pmatrix}
0&1&0&0&\cdots& 0\\
0&0&1&0&\cdots& 0\\
0&0&0&1&\cdots& 0\\
\vdots&\vdots&\vdots&\vdots&\ddots &\vdots\\
0&0&0&0&\cdots&1\\
1&a_{m-1}&a_{m-2}&a_{m-3}&\cdots &a_1
\end{pmatrix}.
\]
As $\mathrm{det}(g)=1$, we deduce $g\in\mathrm{SL}_m(q)$ and, as $p(x)$ is irreducible, we deduce that $\langle g\rangle$ acts irreducibly as a matrix group on $\mathbb{F}_q^m$. A computation gives
\[
g^{-1}:=\begin{pmatrix}
-a_{m-1}&-a_{m-2}&\cdots&-a_1&1\\
1&0&0&\cdots& 0\\
0&1&0&\cdots& 0\\
\vdots&\vdots&\ddots&\vdots &\vdots\\
0&0&\cdots&1&0
\end{pmatrix}.
\]

Suppose $q$ odd. Consider the block diagonal matrix
\[
z:=\begin{pmatrix}
-I_{\frac{m-1}{2}}&&\\
&1&\\
&&-I_{\frac{m-1}{2}}
\end{pmatrix},
\]
where $I_{(m-1)/2}$ is the identity matrix with $(m-1)/2$ rows and columns.
Observe that $z$ is an involution of $\mathrm{SL}_m(q)$ and $z\notin{\bf Z}(\mathrm{SL}_m(q))$, because $m\ge 3$.

We have
\[
[g,z]=g^{-1}zgz=
\begin{pmatrix}
I_{\frac{m-1}{2}}&  &2a_{(m-1)/2}  &\\
                 &-1&  &\\
                 &  &-1&\\
                 &  &  &I_{\frac{m-3}{2}}

\end{pmatrix}.
\]
In particular, $[g,z]$ is an involution.

Suppose $q$ even. Here, an entirely similar construction works simply by taking $z$ to be the transvection
\[
z=
\begin{pmatrix}
1&1&\\
0&1&\\
 & &I_{m-2}
\end{pmatrix}.
\]
A straightforward computation yields that
\[
[g,z]=
\begin{pmatrix}
1&1&a_{m-1}&\\
0&1&1&\\
0&0&1&\\
& & &I_{m-3}
\end{pmatrix}
\]
and hence $[g,z]$ has order $4$.
\end{proof}
\begin{proposition}\label{proposition:PSL}
Let $m\ge 3$ and let $q$ be a prime power. Then $\Gamma_2(\mathrm{PSL}_m(q))$ is strongly connected when $q$ is odd and $\Gamma_3(\mathrm{PSL}_m(q))$ is strongly connected when $q$ is even.
\end{proposition}
\begin{proof}
Let $G:=\mathrm{PSL}_m(q)$. When $(m,q)=(3,2)$, we have $\mathrm{PSL}_3(2)\cong\mathrm{PSL}_2(7)$. It follows from Section~\ref{sec:lie1} that 
$\Gamma_n(G)$ is strongly connected when $n\ge 3$. It can be verified with \texttt{magma} that $\Gamma_2(G)$ has 9 strongly connected components and $\Gamma_1(G)$ has 37 strongly connected components. Similarly, when $(m,q)=(3,4)$, it can be verified with \texttt{magma} that $\Gamma_1(G)$ has 3257 strongly connected components, $\Gamma_2(G)$ has 961 strongly connected components and $\Gamma_3(G)$ is strongly connected. In particular, for the rest of our argument, we do exclude $(m,q)\in \{(3,2),(3,4)\}$ from further considerations.

By Theorem~\ref{theorem111}, all elements of even order of $G$ are in the same connected component of $\Gamma_1(G)$. In particular, we use the prime graph of $G$ to deduce properties of the commuting graph of $G$, see Corollary~\ref{cor}. When the prime graph of $G$ is connected, then so is $\Gamma_1(G)$ and hence so is $\Gamma_n(G)$. Therefore, we may suppose that $\Pi(G)$ is disconnected. We have reported in Table~\ref{table1PSL} informations on the connected components of $\Pi(G)$. This information is taken from~\cite{KoMa}, adapted to our current notation.
\begin{table}\centering
\begin{tabular}{cll}
\toprule[1.5pt]
conditions&nr. components & components\\
\midrule[1.5pt]
$(m,q)=(3,4)$&4&$\{2\}$, $\{3\}$,  $\{5\}$,  $\{7\}$\\ 
$m$ odd prime, $(m,q)\notin\{(3,2),(3,4)\}$ &2&$\pi\left(q\prod_{i=1}^{m-1}(q^i-1)\right)$, $\pi\left(\frac{q^m-1}{(q-1)\gcd(m,q-1)}\right)$\\
$m-1$ odd prime, $q-1\mid m$&2&$\pi\left(q(q^m-1)\prod_{i=1}^{m-2}(q^i-1)\right)$, $\pi\left(\frac{q^{m-1}-1}{q-1}\right)$\\
\bottomrule[1.5pt]
\end{tabular}
\caption{Cases when $\mathrm{PSL}_m(q)$ has disconnected prime graph}\label{table1PSL}
\end{table}

We deal with each line in Table~\ref{table1PSL} in turn. When $m$ is prime, Lemma~\ref{PSL} guarantees that there exists an element $g$ having order divisible by an element in the second connected component of $\Pi(G)$ and an involution 
$z$ with $g\mapsto_2 z$ when $q$ is odd and with $g\mapsto_3z$ when $q$ is even. Conversely, let $g\in \mathrm{PSL}_m(q)$ with ${\bf o}(g)=(q^m-1)/[(q-1)\gcd(m,q-1)]$. Then ${\bf N}_G(\langle g\rangle)=\langle g\rangle\rtimes \langle z\rangle$ where ${\bf o}(z)=m$. In particular, $[z,g,g]=1$ and $z\mapsto_2g$.

Assume now that $m-1$ is prime and $q-1\mid m$. Lemma~\ref{PSL} (applied with $m$ replaced by $m-1$) guarantees that, there exists an element $g$ having order divisible by an element in the second connected component of $\Pi(G)$ and an involution 
$z$ with $g\mapsto_2 z$ when $q$ is odd and with $g\mapsto_3z$ when $q$ is even. Conversely, let $g\in \mathrm{PSL}_m(q)$ with ${\bf o}(g)=(q^{m-1}-1)/(q-1)$. Then ${\bf N}_G(\langle g\rangle)=\langle g\rangle\rtimes \langle z\rangle$ where ${\bf o}(z)=m-1$. In particular, $[z,g,g]=1$ and $z\mapsto_2g$.
\end{proof}

\section{Unitary groups}\label{sec:unitary}

Before embarking into the study of the Engel graph of unitary groups, we need to prove the following.
\begin{theorem}\label{theorem:PSU}
Let $m$ be an odd prime and let $q$ be a prime power. If $(m,q)\ne (3,2)$, then there exists an element  $g\in \mathrm{SU}_m(q)\setminus{\bf Z}(\mathrm{SU}_m(q))$ having order a divisor of $(q^m+1)/(q+1)$ and an element $z\in \mathrm{SU}_m(q)\setminus{\bf Z}(\mathrm{SU}_m(q))$ having order a divisor of $q+1$ such that $[g,z,z]=1$.
\end{theorem}
To prove Theorem~\ref{theorem:PSU}, we recall some observations of Huppert~\cite{H}. As usual, let $\mathbb{F}_{q^{2m}}$ be the finite field having $q^{2m}$ elements. The function 
\begin{alignat*}{2}
\varphi:\mathbb{F}_{q^{2m}}\times\mathbb{F}_{q^{2m}}&\longrightarrow&\mathbb{F}_{q^{2m}}\\ 
(x,y)&\longmapsto& xy^{q^m}
\end{alignat*} is a Hermitian form for the vector space $\mathbb{F}_{q^{2m}}$, viewed as a $1$-dimensional vector space over the field $\mathbb{F}_{q^{2m}}$. Let $a\in\mathbb{F}_{q^{2m}}^\ast$. Then, the function 
\begin{alignat*}{2}
\rho_a:\mathbb{F}_{q^{2m}}&\longrightarrow&\mathbb{F}_{q^{2m}}\\
x&\longmapsto &ax
\end{alignat*} lies in the general unitary group with respect to the Hermitian form $\varphi$ if and only if $$a^{q^m+1}=1,$$ that is, ${\bf o}(a)$ divides $q^m+1$.

Now, consider the trace function 
\begin{alignat*}{2}
\mathrm{Tr}:\mathbb{F}_{q^{2m}}&\longrightarrow&\mathbb{F}_{q^2}\\
x&\longmapsto &\sum_{i=0}^{m-1}x^{q^{2i}}
\end{alignat*}
 with respect to the Galois extension $\mathbb{F}_{q^{2m}}/\mathbb{F}_{q^2}$. Then, Huppert~\cite{H} shows that the function 
\begin{alignat}{2}\label{defpsi}
\psi:\mathbb{F}_{q^{2m}}\times\mathbb{F}_{q^{2m}}&\longrightarrow&\mathbb{F}_{q^2}\\
(x,y)&\longmapsto &\mathrm{Tr}(xy^{q^m})\nonumber
\end{alignat} is a non-degenerate Hermitian form for the vector space $\mathbb{F}_{q^{2m}}$, viewed as an $m$-dimensional vector space over the field $\mathbb{F}_{q^{2}}$. Clearly, $\rho_a$ lies in the general unitary group with respect to the Hermitian form $\psi$, what is more, $\rho_a$ lies in the special unitary group if and only if
 $$a^{\frac{q^m+1}{q+1}}=1  \text { and }a \notin \mathbb{F}_{q^{2}} \text { or }
a^m=1\text { and } a \in \mathbb{F}_{q^{2}}
$$
that is, ${\bf o}(a)$ divides $(q^m+1)/(q+1)$ and ${\bf o}(a)\nmid \gcd(m,q+1).$

In the proof of the next result, we omit the assumption that $m$ is prime, because it is not necessary.
\begin{lemma}\label{lemma:unitary2}Let $q$ be a prime power and let $m$ be an odd integer with $(m,q)\ne (3,2)$. There exists $y\in \mathbb{F}_{q^m}$ and $a\in\mathbb{F}_{q^{2m}}^\ast$ with ${\bf o}(a)\mid (q^m+1)/(q+1)$ and ${\bf o}(a)\nmid \gcd(m,q+1)$ such that $\mathrm{Tr}(y)\ne 0$ and $\mathrm{Tr}(ay)=0$.
\end{lemma}
\begin{proof}
Since $m$ is odd, $$\mathbb{F}_{q^{2m}}=\mathbb{F}_{q^2}\otimes_{\mathbb{F}_q}\mathbb{F}_{q^m}.$$ Let $1,\varepsilon\in\mathbb{F}_{q^2}$ be an $\mathbb{F}_q$-basis of $\mathbb{F}_{q^2}$. In the quadratic extension $\mathbb{F}_{q^2}/\mathbb{F}_q$, there exist $q+1$ elements having norm $1$. In particular, as $q+1>|\mathbb{F}_q|$, replacing $\varepsilon$ with a suitable element, we may suppose that $\varepsilon$ has norm $1$ with respect to the quadratic extension $\mathbb{F}_{q^2}/\mathbb{F}_q$, that is, 
\begin{align}\label{eq:silly0}\varepsilon^{q+1}=1.
\end{align} (This is not very relevant for the rest of the proof, but it makes a computation later easier to follow.) 

Therefore, we may write every element $x\in \mathbb{F}_{q^{2m}}$ in the form $$x=x_1+x_2\varepsilon,$$ for a unique choice of $x_1,x_2\in \mathbb{F}_{q^m}$.
In what follows, for every $x\in\mathbb{F}_{q^{2m}}$, the elements $x_1,x_2\in\mathbb{F}_{q^m}$ denote the unique elements with $x=x_1+x_2\varepsilon.$

Let $x\in \mathbb{F}_{q^m}$. Now, $x^{q^m}=x$ and hence $x^{q^{m+i}}=x^{q^i}$, for each $i$. Therefore, using that $m$ is odd, we have
\begin{align*}
\mathrm{Tr}(x)&=x+x^{q^2}+x^{q^4}+\cdots +x^{q^{2(m-1)}}\\
&=x+x^{q^2}+x^{q^4}+\cdots+x^{q^{m-1}}+x^{q^{m+1}}+x^{q^{m+3}}+\cdots+ x^{q^{m+(m-2)}}\\
&=x+x^{q^2}+x^{q^4}+\cdots+x^{q^{m-1}}+x^{q}+x^{q^3}+\cdots+ x^{q^{m-2}}\\
&=x+x^q+\cdots+x^{q^{m-1}}\\
&=\mathrm{Tr}_{\mathbb{F}_{q^m}/\mathbb{F}_q}(x).
\end{align*}
In other words, if we denote with $\mathrm{tr}:\mathbb{F}_{q^m}\to \mathbb{F}_q$ the trace map $\mathrm{Tr}_{\mathbb{F}_{q^m}/\mathbb{F}_q}$, we obtain
\begin{align*}
\mathrm{Tr}(x)=\mathrm{tr}(x),\quad\forall x\in\mathbb{F}_{q^m}.
\end{align*}

Now, the trace map $\mathrm{Tr}:\mathbb{F}_{q^{2m}}\to \mathbb{F}_{q^2}$ is obtained by tensoring the identity map $\mathrm{id}:\mathbb{F}_{q^2}\to \mathbb{F}_{q^2}$ with $\mathrm{tr}$. Indeed,
\begin{align}\label{eq:silly}
\mathrm{Tr}(x)&=\mathrm{Tr}(x_1+x_2\varepsilon)=\mathrm{Tr}(x_1)+\mathrm{Tr}(x_2)\varepsilon=\mathrm{tr}(x_1)+\mathrm{tr}(x_2)\varepsilon.
\end{align}
In particular, $\mathrm{Tr}(x)=0$ if and only if $\mathrm{tr}(x_1)=\mathrm{tr}(x_2)=0$.

Let 
$$\mathcal{A}:=\left\{a\in\mathbb{F}_{q^{2m}}^\ast: {\bf o}(a)\mid \frac{q^m+1}{q+1} \hbox{ and }{\bf o}(a)\nmid \gcd(m,q+1)\right\}.$$ From~\eqref{eq:silly}, given  $y\in \mathbb{F}_{q^m}$ and $a\in\mathcal{A}$, we have $\mathrm{Tr}(y)\ne 0$ and  $\mathrm{Tr}(ay)=0$ if and only if $\mathrm{tr}(y)\ne 0$ and $\mathrm{tr}(a_1y)=\mathrm{tr}(a_2y)=0$.

Before proceeding further in the proof, we prove a preliminary fact.

\smallskip

\noindent\textsc{Claim:} Given $x=x_1+x_2\varepsilon\in\mathbb{F}_{q^{2m}}$, there exists $y\in\mathbb{F}_{q^m}$ with 
$\mathrm{tr}(y)\ne 0$ and $\mathrm{tr}(x_1y)=\mathrm{tr}(x_2y)=0$ if and only if $1$ is not an $\mathbb{F}_q$-linear combination of $x_1$ and $x_2$. 

\smallskip

\noindent Suppose first that $1$ is an $\mathbb{F}_q$-linear combination of $x_1$ and $x_2$, that is, $1=\lambda_1 x_1+\lambda_2 x_2$, for some $\lambda_1,\lambda_2\in \mathbb{F}_q$. For every $y\in\mathbb{F}_{q^m}$, we have $y=\lambda_1 x_1y+\lambda_2 x_2y$ and hence
$$\mathrm{tr}(y)=\mathrm{tr}(\lambda_1 x_1y+\lambda_2 x_2y)=\lambda_1\mathrm{tr}(x_1y)+\lambda_2\mathrm{tr}(x_2y).$$ $\mathrm{tr}(yx_1)=\mathrm{tr}(yx_2)=0$, then  $\mathrm{tr}(y)=\mathrm{tr}(\lambda_1 y x_1+\lambda_2 y x_2)=
\lambda_1 \mathrm{tr}(y x_1)+\lambda_2 \mathrm{tr}(y x_2)=0$. In particular, there is no $y\in\mathbb{F}_{q^m}$ with $\mathrm{tr}(y)\ne 0$ and $\mathrm{tr}(x_1y)=\mathrm{tr}(x_2y)=0$.

Suppose next that, for every $y\in \mathbb{F}_{q^m}$ with $\mathrm{tr}(x_1y)=\mathrm{tr}(x_2y)=0$, we have $\mathrm{tr}(y)=0$. Recall now that the mapping
\begin{alignat*}{3}
\theta:\mathbb{F}_{q^{m}}&\longrightarrow&\mathrm{Hom}_{\mathbb{F}_{q}}(\mathbb{F}_{q^m},\mathbb{F}_q)&&\\
z&\longmapsto &\theta_z:\mathbb{F}_{q^m}&\longrightarrow&\mathbb{F}_q\\
&&\zeta&\longmapsto \zeta z
\end{alignat*}
is an $\mathbb{F}_q$-isomorphism between $\mathbb{F}_{q^m}$ and the dual space $\mathbb{F}_{q^m}$. This isomorphism can be used to construct a non-degenerate bilinear form
\begin{alignat*}{2}
\langle\cdot,\cdot\rangle_\theta:\mathbb{F}_{q^{m}}\times\mathbb{F}_{q^m}&\longrightarrow\mathbb{F}_{q^m}\\
(z_1,z_2)&\longmapsto &\langle z_1,z_2\rangle_\theta=\mathrm{tr}(z_1z_2).
\end{alignat*}
If $1$ is not a linear combination of $x_1$ and $x_2$, then $1\notin\mathrm{Span}_{\mathbb{F}_{q}}(x_1,x_2)$ and hence by taking the orthogonal complements we deduce $\mathrm{Span}_{\mathbb{F}_q}(x_1,x_2)^{\perp_\theta}\nleq \mathrm{Span}_{\mathbb{F}_q}(1)^{\perp_\theta}$. In other words, there exists $y\in \mathrm{Span}_{\mathbb{F}_q}(x_1,x_2)^{\perp_\theta}$ with $y\notin\mathrm{Span}_{\mathbb{F}_q}(1)^{\perp_\theta}$, that is, $\mathrm{tr}(x_1y)=\mathrm{tr}(x_2y)=0$ and $\mathrm{tr}(y)\ne 0$. However, this is contrary to our assumption and hence $1$ is a linear combination of $x_1$ and $x_2$.~$_\blacksquare$ 

\smallskip

In what follows, we apply the previous claim when $x$ is an element in $\mathcal{A}$. Indeed, from the previous claim, to conclude the proof of this lemma, it suffices to prove that there exists $a=a_1+a_2\varepsilon\in\mathcal{A}$ such that $1$ is not an $\mathbb{F}_q$-linear combination of $a_1$ and $a_2$.

 Let $\mathcal{B}$ be the collection of all elements $a=a_1+a_2\varepsilon\in\mathcal{A}$ such that $1$ is an $\mathbb{F}_q$-linear combination of $a_1$ and $a_2$. We now show that  $\mathcal{B}$ is strictly contained in $\mathcal{A}$, that is, $\mathcal{B}\subsetneq \mathcal{A}$.

As $a^{q^m+1}=1$ for every $a\in\mathcal{A}$, the set $\mathcal{A}$ is contained in the subgroup of order $q^m+1$ of the multiplicative group of $\mathbb{F}_{q^{2m}}$.  Thus the norm, with respect to the quadratic extension $\mathbb{F}_{q^{2m}}/\mathbb{F}_{q^m}$, of each element $a=a_1+a_2\varepsilon$ of $\mathcal{A}$ is $1$. Using~\eqref{eq:silly0}, we get
\begin{align}\label{eq:silly2}1&=(a_1+a_2\varepsilon)(a_1+a_2\varepsilon)^{q^m}=a_1^2+a_2^2\varepsilon^{q+1}+a_1a_2(\varepsilon+\varepsilon^q)\nonumber\\
&=a_1^2+a_2^2+a_1a_2(\varepsilon+\varepsilon^q).
\end{align}
Observe that~\eqref{eq:silly2} is a non-singular quadratic equation in $a_1,a_2\in\mathbb{F}_{q^m}$. 

For each $\lambda_1,\lambda_2\in\mathbb{F}_q$, not both zero, $$\lambda_1a_1+\lambda_2a_2=1$$
is a linear equation in $a_1,a_2\in\mathbb{F}_{q^m}$. There can be at most two solutions to a linear equation together with a non-singular quadratic equation in two variables, since (after substituting) it gives a nontrivial quadratic equation in one variable.
Summing over all possible choices of $\lambda_1,\lambda_2$, the number of elements $a$ of $\mathcal{A}$ such that $1$ is a linear combination of $a_1$ and $a_2$ is at most $2(q^2-1)$. So we can only have all elements of $\mathcal{A}$
with this property if $$\frac{q^m+1}{q+1}-\gcd(m,q+1)=|\mathcal{A}|\le 2(q^2-1).$$ However,  this is possible only when $m=3$. This concludes the proof of the lemma when $m\ne 3$.

Suppose then $m=3$. Here we simply refine the argument above. Let $\tilde{\mathcal{A}}$ be the subgroup of the multiplicative group of $\mathbb{F}_{q^6}$ having order $q^3+1$, let $\mathcal{A}'$ be the subgroup of $\tilde{\mathcal{A}}$ having order $q+1$ and we redefine $\mathcal{A}$ as the subgroup of the multiplicative group of $\mathbb{F}_{q^6}$ having order $(q^3+1)/(q+1)=q^2-q+1$. Since $\tilde{\mathcal{A}}$ is cyclic, $\mathcal{A}\circ\mathcal{A}'$ is a subgroup of $\tilde{\mathcal{A}}$ 
 with $|\tilde{\mathcal{A}}:\mathcal{A}\circ\mathcal{A}'|=\gcd(3,q+1).$ 

Observe that, when $q\notin\{2,5\}$, we have $$|\mathcal{A}\circ\mathcal{A}'|=\frac{q^3+1}{\gcd(3,q+1)}>2(q^2-1)$$ and hence the argument above shows that there exists an element $a'\in\mathcal{A}\circ\mathcal{A}'$ and an element $y\in\mathbb{F}_{q^3}$ with ${\bf o}(a')\mid q^3+1$, $\mathrm{Tr}(y)\ne 0$ and $\mathrm{Tr}(a'y)=0$. Now, multiplying $a'$ by a suitable element $\lambda\in \mathbb{F}_{q^2}^\ast$ we have ${\bf o}(\lambda a')\mid q^2-q+1=(q^3+1)/(q+1)$. Set $a:=a'\lambda$. Then ${\bf o}(a)\mid (q^3+1)/(q+1)$ and
$$\mathrm{Tr}(ay)=\mathrm{Tr}(a'\lambda y)=\lambda\mathrm{Tr}(a'y)=\lambda\cdot 0=0.$$
If ${\bf o}(a)\mid \gcd(3,q+1)$, then $a\in\mathbb{F}_{q^2}$ and hence
$$0=\mathrm{Tr}(ay)=a\mathrm{Tr}(y)\ne 0,$$
because $\mathrm{Tr}(y)\ne 0$.
Therefore, this proves the lemma when $q\notin\{3,5\}$. Finally, when $q=5$, the result can be checked with the auxiliary help of a computer and, when $q=2$, we obtain the exceptional case excluded in the statement of the lemma.
\end{proof}

\begin{lemma}\label{lemma:unitary3}Let $q$ be an even prime power and let $m$ be an odd integer with $m\ge 5$ and $m=q+1$. There exists $y\in \mathbb{F}_{q^m}$ and $a\in\mathbb{F}_{q^{2m}}^\ast$ with ${\bf o}(a)\mid (q^m+1)/(q+1)$ and ${\bf o}(a)\nmid \gcd(m,q+1)$ such that $\mathrm{Tr}(y)\ne 0$, $\mathrm{Tr}(ay)=0$ and $\mathrm{Tr}(a^2y)=0$.
\end{lemma}
\begin{proof}
The proof follows verbatim the proof of Lemma~\ref{lemma:unitary2}.
\end{proof}

Now, we apply Lemmas~\ref{lemma:unitary2} and~\ref{lemma:unitary3} to deduce some properties on non-degenerate Hermitian spaces. Recall that the mapping $\psi$ defined in~\eqref{defpsi} endows $\mathbb{F}_{q^{2m}}$ of the structure of a non-degenerate Hermitian space over $\mathbb{F}_{q^2}$. In particular, $v\in\mathbb{F}_{q^{2m}}$ is non-degenerate when $\psi(v,v)\ne 0$, $v$ is isotropic when $\psi(v,v)=0$ and $v^\perp:=\{w\in\mathbb{F}_{q^{2m}}\mid \psi(v,w)=0\}$. For not making the notation too cumbersome, we identify $\rho_a$ with $a$ itself.

\begin{lemma}\label{lemma:unitary1}
Let $q$ be a prime power and let $m$ be an odd integer with $(m,q)\ne (3,2)$. There exists a non-degenerate vector $v\in \mathbb{F}_{q^{2m}}$ and an element $a\in\mathbb{F}_{q^{2m}}^\ast$  with ${\bf o}(a)\mid (q^m+1)/(q+1)$ and ${\bf o}(a)\nmid \gcd(m,q+1)$ such that 
$va\in v^\perp$.
\end{lemma}
\begin{proof}
Let $y\in\mathbb{F}_{q^m}^\ast$ and $a\in\mathbb{F}_{q^{2m}}^\ast$ be as in the conclusion of Lemma~\ref{lemma:unitary2}.

Since $\mathbb{F}_{q^{2m}}^\ast$ is cyclic of order $q^{2m}-1$ and $\mathbb{F}_{q^m}^\ast$ is cyclic of order $q^m-1$, there exists $v\in\mathbb{F}_{q^{2m}}^\ast$ with $v^{q^{m}+1}=y$. Now, we have
\begin{align*}
\psi(v,v)&=\mathrm{Tr}(vv^{q^m})=\mathrm{Tr}(v^{q^m+1})=\mathrm{Tr}(y)\ne 0,\\
\psi(av,v)&=\mathrm{Tr}(avv^{q^m})=\mathrm{Tr}(av^{q^m+1})=\mathrm{Tr}(ay)= 0.
\end{align*}
In particular, $v$ is non-degenerate and $va\in v^\perp$ with respect to the Hermitian form $\psi$.
\end{proof}

\begin{lemma}\label{lemma:unitary1bis}
Let $q$ be an even prime power and let $m$ be an odd integer with $m\ge 5$ and $\gcd(m,q+1)=q+1$. There exists a non-degenerate vector $v\in \mathbb{F}_{q^{2m}}$ and an element $a\in\mathbb{F}_{q^{2m}}^\ast$  with ${\bf o}(a)\mid (q^m+1)/(q+1)$ and ${\bf o}(a)\nmid \gcd(m,q+1)$ such that 
$va\in v^\perp$ and $va^{2}\in v^\perp$.
\end{lemma}
\begin{proof}
Let $y\in\mathbb{F}_{q^m}^\ast$ and $a\in\mathbb{F}_{q^{2m}}^\ast$ be as in the conclusion of Lemma~\ref{lemma:unitary3}.

Since $\mathbb{F}_{q^{2m}}^\ast$ is cyclic of order $q^{2m}-1$ and $\mathbb{F}_{q^m}^\ast$ is cyclic of order $q^m-1$, there exists $v\in\mathbb{F}_{q^{2m}}^\ast$ with $v^{q^{m}+1}=y$. Now, we have
\begin{align*}
\psi(v,v)&=\mathrm{Tr}(vv^{q^m})=\mathrm{Tr}(v^{q^m+1})=\mathrm{Tr}(y)\ne 0,\\
\psi(av,v)&=\mathrm{Tr}(avv^{q^m})=\mathrm{Tr}(av^{q^m+1})=\mathrm{Tr}(ay)= 0,\\
\psi(a^2v,v)&=\mathrm{Tr}(a^2vv^{q^m})=\mathrm{Tr}(a^2v^{q^m+1})=\mathrm{Tr}(a^{2}y)= 0.
\end{align*}
In particular, $v$ is non-degenerate and $va,va^{2}\in v^\perp$ with respect to the Hermitian form $\psi$.
\end{proof}

\begin{proof}[Proof of Theorem~$\ref{theorem:PSU}$]
Let $v\in\mathbb{F}_{q^{2m}}$ and $a\in\mathbb{F}_{q^{2m}}^\ast$ be as in the conclusion of Lemma~\ref{lemma:unitary1}. Then $a\in \mathrm{SU}_m(q)\setminus {\bf Z}(\mathrm{SU}_m(q))$ and ${\bf o}(a)\mid (q^m+1)/(q+1)$. 

Assume first that $m\neq q+1$. Let $\alpha\in \mathbb{F}_{q^2}^\ast$ with ${\bf o}(\alpha)=q+1$. Consider the $\mathbb{F}_{q^2}$-linear mapping $z:\mathbb{F}_{q^{2m}}\to\mathbb{F}_{q^{2m}}$ with $wz=\alpha w$, for every $w\in v^\perp$, and $vz=\alpha^{1-m}v$. With respect to a suitable $\mathbb{F}_{q^2}$-basis of $\mathbb{F}_{q^{2m}}$ the element $z$ is represented by the matrix
\[
\begin{pmatrix}
\alpha^{1-m}&&&\\
&\alpha&&\\
&      &\ddots&\\
&&&\alpha
\end{pmatrix}.
\]
It is not hard to verify that $z\in\mathrm{GU}_m(q)$ with respect to the Hermitian form $\psi$. Moreover, as $\det(z)=\alpha^{1-m}\alpha^{m-1}=1$, we deduce $z\in\mathrm{SU}_m(q)$. Now, $z\in{\bf Z}(\mathrm{SU}_m(q))$ if and only if $\alpha^{1-m}=\alpha$, that is $\alpha^m=1$. Since ${\bf o}(\alpha)=q+1$, this happens if and only if $q+1=m$, which we are excluding for the time being.

Now, we fix an orthogonal $\mathbb{F}_{q^2}$-basis $(v_1,\ldots,v_m)$ of $\mathbb{F}_{q^{2m}}$ with respect to $\psi$. Without loss of generality, we may suppose that $v_1:=v$ and $v_2:=va$, because $va\in v^\perp$. We now determine the matrix representation of $[a,z]$ with respect to this basis. Observe that $v=va (a^{-1})\in (v^\perp)a^{-1}=(va^{-1})^\perp$ and hence $va^{-1}\in v^\perp$. We have
\begin{align*}
v_1[a,z]&=v_1a^{-1}z^{-1}az=(v_1a^{-1})z^{-1}az=\alpha^{-1} (v_1a^{-1})az=\alpha^{-1} v_1z\\
&=\alpha^{-1} \alpha^{1-m}v_1=\alpha^{-m}v_1,\\
v_2[a,z]&=v_2a^{-1}z^{-1}az=v_1z^{-1}az=\alpha^{m-1} v_1a z=\alpha^{m-1} v_2z=\alpha^{m-1} \alpha v_2=\alpha^{m}v_2.
\end{align*}
When $i\in \{3,\ldots,m\}$, observe that $v_i\in v_2^\perp=(v_1a)^\perp=(v_1^\perp)a$ and hence $v_ia^{-1}\in v_1^\perp$. Therefore, we have
\begin{align*}
v_i[a,z]&=v_ia^{-1}z^{-1}az=(v_ia^{-1})z^{-1}az=\alpha^{-1} v_ia^{-1}az=\alpha^{-1} v_iz=\alpha^{-1} \alpha v_i=v_i.
\end{align*}
Therefore, the matrix representing $[a,z]$ is
\[
\begin{pmatrix}
\alpha^{-m}&&&&\\
&\alpha^m&&&\\
&&1&&\\
&&      &\ddots&\\
&&&&1
\end{pmatrix}.
\]
Clearly, $[a,z]$ commutes with $z$ and hence $[a,z,z]=1$. 

\smallskip

Suppose now $m=q+1$. Since $m$ is odd, $q$ is even and hence $q$ is a power of $2$. Observe that $m\ge 5$, because $(m,q)\ne(3,2)$ by hypothesis.  As above, let $v_1,\ldots,v_m$ be an orthogonal $\mathbb{F}_{q^2}$-basis. Using Lemma~\ref{lemma:unitary1bis}, we may suppose that $v_2=v_1a$ and $v_2a=v_3$. Let $\alpha\in\mathbb{F}_{q^2}^\ast$ with ${\bf o}(\alpha)=q+1$. Let $z$ be the $\mathbb{F}_{q^2}$-linear map with $v_1z=\alpha v_1$, $v_2z=\alpha^{-1}v_2$ and $v_iz=z$, for each $i\in \{3,\ldots,m\}$. Then  the element $z$ is represented by the matrix
\[
\begin{pmatrix}
\alpha&0&\\
0&\alpha^{-1}&\\
&&I
\end{pmatrix}.
\]
It is not hard to verify that $z\in\mathrm{GU}_m(q)$ with respect to the Hermitian form $\psi$. Moreover, as $\det(z)=1$, we deduce $z\in\mathrm{SU}_m(q)$. Clearly, $z\notin{\bf Z}(\mathrm{SU}_m(q))$. As $v_1\in \langle v_2,v_3\rangle^\perp$, we have $v_1a^{-1}\in \langle v_2a^{-1},v_3a^{-1}\rangle^\perp=\langle v_1,v_2\rangle^\perp=\langle v_3,\ldots,v_m\rangle$. Thus $v_1a^{-1}z=v_1a^{-1}$. We deduce
\begin{align*}
v_1[a,z]&=v_1a^{-1}z^{-1}az=(v_1a^{-1})z^{-1}az=(v_1a^{-1})az=v_1z=\alpha v_1.
\end{align*}
Similarly, we have
\begin{align*}
v_2[a,z]&=v_2a^{-1}z^{-1}az=v_1z^{-1}az=\alpha^{-1} v_1a z=\alpha^{-1} v_2z=\alpha^{-2} v_2,\\
v_3[a,z]&=v_3a^{-1}z^{-1}az=v_2z^{-1}az=\alpha v_2a z=\alpha v_3 z=\alpha v_3.
\end{align*}
When $i\in \{4,\ldots,m\}$, observe that $v_i\in \langle v_1,v_2,v_3\rangle^\perp$ and hence $z$ fixes $v_i$ and $v_ia^{-1}$. Therefore, we have
\begin{align*}
v_i[a,z]&=v_ia^{-1}z^{-1}az=(v_ia^{-1})z^{-1}az=v_ia^{-1}az=v_iz=v_i.
\end{align*} 
Therefore, the matrix representing $[a,z]$ is
\[
\begin{pmatrix}
\alpha&&&\\
&\alpha^{-2}&&\\
&&\alpha&\\
&&      &I\\
\end{pmatrix}.
\]
Clearly, $[a,z]$ commutes with $z$ and hence $[a,z,z]=1$. 
\end{proof}

\subsection{The Engel graph of unitary groups}\label{subsec:unitary}
\begin{proposition}\label{proposition:PSUUU}
Let $m\ge 3$ and let $q$ be a prime power with $(m,q)\ne (3,2)$. Then $\Gamma_2(\mathrm{PSU}_m(q))$ is strongly connected.
\end{proposition}
\begin{proof}
Let $G:=\mathrm{PSU}_m(q)$. 
By Theorem~\ref{theorem111}, all elements of even order of $G$ are in the same connected component of $\Gamma_1(G)$. In particular, we use the prime graph of $G$ to deduce properties of the commuting graph of $G$, see Corollary~\ref{cor}. When the prime graph of $G$ is connected, then so is $\Gamma_1(G)$ and hence so is $\Gamma_n(G)$. Therefore, we may suppose that $\Pi(G)$ is disconnected. We have reported in Table~\ref{tablePSU} informations on the connected components of $\Pi(G)$. This information is taken from~\cite{KoMa}, adapted to our current notation. (Observe that there is a typo in~\cite[Table~1]{KM}, the authors in line~5 write $(p,q)\ne (3,3),(5,2)$. However, they actually intended to write $(p,q)\ne (3,2),(5,2)$.)
\begin{table}[!ht]\centering
\begin{tabular}{cll}
\toprule[1.5pt]
conditions&nr. components & components\\
\midrule[1.5pt]
$m$ odd prime&2&$\pi\left(q\prod_{i=1}^{m-1}(q^i-(-1)^i)\right)$, $\pi\left(\frac{q^m+1}{(q+1)\gcd(m,q+1)}\right)$\\
$m-1$ odd prime, $q+1\mid m$, &2&$\pi\left(q(q^m-1)\prod_{i=1}^{m-2}(q^i-(-1)^i)\right)$, $\pi\left(\frac{q^{m-1}+1}{q+1}\right)$\\
$(m,q)\notin\{(4,2),(6,2)\}$&&\\
$(m,q)=(4,2)$&2&$\{2,3\}$, $\{5\}$\\
$(m,q)=(6,2)$&3&$\{2,3,5\}$, $\{7\}$, $\{11\}$\\
\bottomrule[1.5pt]
\end{tabular}
\caption{Cases when $\mathrm{PSU}_m(q)$ has disconnected prime graph}\label{tablePSU}
\end{table}

We deal with each line in Table~\ref{tablePSU} in turn. When $m$ is prime, Theorem~\ref{theorem:PSU}  guarantees that, there exists an element $g$ having order divisible by an element in the second connected component of $\Pi(G)$ and an element $z$ in the first connected component of 
$\Pi(G)$ with $g\mapsto_2 z$. Moreover, let $g\in \mathrm{PSU}_m(q)$ with ${\bf o}(g)=(q^m+1)/((q+1)\gcd(m,q+1))$ and let $C:=\langle g\rangle$. Then $C={\bf C}_G(g)$ and ${\bf N}_G(C)=C\rtimes \langle z\rangle$ with ${\bf o}(z)=m$. This gives $z\mapsto_2g$ from which it follows that $\Gamma_2(G)$ is strongly connected.

Assume now that $m-1$ is prime and $q+1\mid m$. Therem~\ref{theorem:PSU} (applied with $m$ replaced by $m-1$) guarantees that, there exists an element $g$ having order divisible by an element in the second connected component of $\Pi(G)$ and an element $z$ in the first connected component of 
$\Pi(G)$ with $g\mapsto_2 z$. Moreover, let $g\in \mathrm{PSU}_m(q)$ with ${\bf o}(g)=(q^{m-1}+1)/(q+1)$ and let $C:=\langle g\rangle$. Then $C={\bf C}_G(g)$ and ${\bf N}_G(C)=C\rtimes \langle z\rangle$ with ${\bf o}(z)=m-1$. This gives $z\mapsto_2g$ from which it follows that $\Gamma_2(G)$ is strongly connected.

When $(m,q)=(4,2)$, a computer computation shows that $\Gamma_1(G)$ has $1297$ strongly connected components and $\Gamma_2(G)$ is strongly connected.

When $(m,q)=(6,2)$, it can be checked that for every $g\in G$ with ${\bf o}(g)=7$, there exist involutions $z_1$ and $z_2$ with $g\mapsto_2 z_1$ and $z_2\mapsto_2 g$. Similarly, it can be checked that, for every $g\in G$ with ${\bf o}(g)=11$, there exists an element $z_1$ of order $5$ with $z_1\mapsto_2 g$ (this is actually clear from the fact that ${\bf N}_G(\langle g\rangle)$ is a Frobenius group of order $55$)  and an element $z_2$ of order $3$ with $g\mapsto_2 z_2$. From this it easy follows that $\Gamma_2(G)$ is strongly connected.
\end{proof}

\section{The Engel graph of symplectic groups}\label{sec:symplectic}
Using our analysis of the Engel graph of linear and unitary groups, our investigation for the remaining classical groups is easier.

Let $q$ be a prime power and let $m$ be a positive integer with $m\ge 2$.  We have reported in Table~\ref{tablePSp} informations on the connected components of $\Pi(\mathrm{PSp}_{2m}(q))$, when it consists of at least two connected components. As usual, this information is taken from~\cite{KoMa}, adapted to our current notation.
\begin{table}[!ht]\centering
\begin{tabular}{cll}
\toprule[1.5pt]
conditions&nr. components & components\\
\midrule[1.5pt]
$m=2^\ell\ge 2$, &2&$\pi\left(q\prod_{i=1}^{m-1}(q^{2i}-1)\right)$, $\pi\left(\frac{q^{m}+1}{\gcd(2,q-1)}\right)$\\
$m$ prime, $q\in \{2,3\}$  &2&$\pi\left(q(q^m+1)\prod_{i=1}^{m-1}(q^{2i}-1)\right)$, $\pi\left(\frac{q^{m}-1}{\gcd(2,q-1)}\right)$\\
\bottomrule[1.5pt]
\end{tabular}
\caption{Cases when $\mathrm{PSp}_{2m}(q)$ has disconnected prime graph}\label{tablePSp}
\end{table}

\begin{proposition}\label{proposition:symp} Let $q$ be a prime power and let $m$ be a positive integer with $m\ge 2$. Then $\Gamma_3(\mathrm{PSp}_{2m}(q))$ is strongly connected. (Recall that, when $q=2$, the group $\mathrm{Sp}_4(2)$ is not simple and $\mathrm{Sp}_4(2)'\cong\mathrm{Alt}(6)$.)
\end{proposition}
\begin{proof}
Let $G:=\mathrm{PSp}_{2m}(q)$. 
By Theorem~\ref{theorem111}, all elements of even order of $G$ are in the same connected component of $\Gamma_1(G)$. In particular, we use the prime graph of $G$ to deduce properties of the commuting graph of $G$, see Corollary~\ref{cor}. When the prime graph of $G$ is connected, then so is $\Gamma_1(G)$ and hence so is $\Gamma_n(G)$. Therefore, we may suppose that $\Pi(G)$ is disconnected and we may use Table~\ref{tablePSp}, which is taken from~\cite{KoMa}.

We deal with each line in Table~\ref{tablePSp} in turn. Suppose $m=2^\ell\ge 2$. We have $$\mathrm{PSp}_{2m}(q)\ge \mathrm{PSp}_2(q^{m})\cong \mathrm{PSL}_2(q^m).$$ Now, as $m\ge 2$, we may deduce the strongly connectivity of the Engel graph of $\mathrm{PSp}_{2m}(q)$ from the strongly connectivity of the Engel graph of $\mathrm{PSL}_{2}(q^m)$, except when $q$ is even. Indeed, when $q$ is odd, we have $q^2\equiv 1\pmod 8$. In particular, when $q\ne 9$, Theorem~\ref{thrm:tired} implies that $\Gamma_2(\mathrm{PSL}_2(q^m))$ is strongly connected. From this it immediately follows that $\Gamma_2(\mathrm{PSp}_{2m}(q))$ is also strongly connected. When $q=9$, it follows from a computer computation that $\Gamma_2(\mathrm{PSp}_4(q))$ is strongly connected.

Assume then $q$ even. We have $$\mathrm{PSp}_{2m}(q)=\mathrm{Sp}_{2m}(q)\ge \mathrm{Sp}_4(q^{m/2}).$$ In particular, arguing as above, we may deduce the strongly connectivity of $\Gamma_2(\mathrm{PSp}_{2m}(q))$ from the strongly connectivity of $\Gamma_2(\mathrm{Sp}_4(q^{m/2}))$. Therefore, without loss in generality, we may suppose that $G=\mathrm{Sp}_4(q)$. We start with $\mathrm{Sp}_2(q^2)=\mathrm{SL}_2(q^2)$. Here, we fix the matrix
\[
J:=\begin{pmatrix}0&1\\1&0\end{pmatrix}\in\mathrm{SL}_2(q).
\]
With respect the symplectic form given by the matrix $J$, the corresponding symplectic group is simply $\mathrm{SL}_2(q^2)$. In particular, the symplectic form $\varphi:\mathbb{F}_{q^2}^2\times\mathbb{F}_{q^2}^2\to\mathbb{F}_{q^2}$ is defined by 
$$\varphi((a,b),(c,d)):=ad+bc.$$
We consider 
$$g=\begin{pmatrix}0&1\\1&\alpha\end{pmatrix}\in \mathrm{SL}_2(q^2)$$
having order $q^2+1$. Composing $\varphi$ with the trace $\mathrm{Tr}:\mathbb{F}_{q^2}\to\mathbb{F}_q$, we obtain a non-degenerate symplectic form $\psi$ for $\mathbb{F}_q^4$. With respect to this form, we have 
$$g=
\begin{pmatrix}
0&0&1&0\\
0&0&0&1\\
1&0&a&b\\
0&1&b&d\end{pmatrix}\in \mathrm{Sp}_4(q).
$$
Moreover, consider
$$z=
\begin{pmatrix}
0&1&0&0\\
1&0&0&0\\
0&0&0&1\\
0&0&1&0\end{pmatrix}.
$$
Observe that $z\in\mathrm{Sp}_4(q)$ and that $z^2=1$. We have
$$[g,z]=g^{-1}zgz=
\begin{pmatrix}
1&0&a+d&b+c\\
0&1&b+c&a+d\\
0&0&1&0\\
0&0&0&1
\end{pmatrix}$$
and
$$[g,z,z]=1.$$
Thus $g\mapsto_2z$. On the other hand, the normalizer in $\mathrm{Sp}_4(q)$
of $\langle g \rangle$ has order $4(q+1)$ (the factor 4 comes from the Galois group of $\mathbb{F}_{q^4}$ over $\mathbb{F}_{q}$), so $z^\prime \mapsto g$ if $z^\prime$ is an involution in this normalizer. From this it easily follows that $\Gamma_2(\mathrm{Sp}_4(q))$ is strongly connected.

\smallskip

We now deal with the second line of Table~\ref{tablePSp}. Suppose $m$ is prime and $q\in \{2,3\}$. The second connected component of $\Pi(G)$ consists of the prime divisors of the order of a torus $T=\langle g\rangle$ having order $(q^m-1)/\gcd(2,q-1)$. This torus acts on the underlying vector space $\mathbb{F}_q^{2m}$ and, in this action, $\mathbb{F}_q^{2m}=V_1\oplus V_2$, where $V_1$ and $V_2$ are two irreducible $\mathbb{F}_qT$-modules. The subspaces $V_1$ and $V_2$ are totally isotropic and, if we denote with $a_i$ the matrix induced by $g$ on $V_i$, then $a_2=(a_1^{-1})^{\mathrm{tr}}$. We have $$\mathrm{Sp}_{2m}(q)\ge \mathrm{GL}_m(q).2,$$ where $\mathrm{GL}_m(q).2$ is the stabilizer of the direct sum decomposition $V_1\oplus V_2$. Therefore, when $m\ne 2$, we may deduce the strongly connectivity of $\Gamma_3(G)$ from the strongly connectivity of $\Gamma_3(\mathrm{PSL}_m(q))$, see Proposition~\ref{proposition:PSL}. When $m=2$, $\Gamma_3(\mathrm{PSp}_4(2)')$ and $\Gamma_3(\mathrm{PSp}_4(3))$ are both strongly connected.
\end{proof}

\section{The Engel graph of orthogonal groups}\label{sec:orthogonal}

\begin{proposition}\label{proposition:omega}
Let $m\ge 3$ be an integer and let $q$ be an odd prime power. The directed graph $\Gamma_2(\Omega_{2m+1}(q))$ is strongly connected.
\end{proposition}
\begin{proof}
Let $G:=\Omega_{2m+1}(q)$. 
By Theorem~\ref{theorem111}, all elements of even order of $G$ are in the same connected component of $\Gamma_1(G)$. In particular, we use the prime graph of $G$ to deduce properties of the commuting graph of $G$, see Corollary~\ref{cor}. When the prime graph of $G$ is connected, then so is $\Gamma_1(G)$ and hence so is $\Gamma_n(G)$. Therefore, we may suppose that $\Pi(G)$ is disconnected and we may use Table~\ref{tableOmega}.

\begin{table}[!ht]\centering
\begin{tabular}{cll}
\toprule[1.5pt]
conditions&nr. components & components\\
\midrule[1.5pt]
$m=2^\ell\ge 2$, $q$ odd &2&$\pi\left(q\prod_{i=1}^{m-1}(q^{2i}-1)\right)$, $\pi\left(\frac{q^{m}+1}{2}\right)$\\
$m$ prime, $q=3$ , &2&$\pi\left(q(q^m+1)\prod_{i=1}^{m-1}(q^{2i}-1)\right)$, $\pi\left(\frac{q^{m}-1}{2}\right)$\\
\bottomrule[1.5pt]
\end{tabular}
\caption{Cases when $\Omega_{2m+1}(q)$ has disconnected prime graph}\label{tableOmega}
\end{table}

We deal with each line in Table~\ref{tableOmega} in turn. Suppose $m=2^\ell\ge 2$. We have 
$$
\Omega_{2m+1}(q)\ge
\Omega_{2m}^-(q)= 
\Omega_{2^{\ell+1}}^-(q)\ge \Omega_{2^\ell}^-(q^2)\ge\cdots\ge \Omega_4^-(q^{m/2})\cong\mathrm{PSL}_2(q^m).$$
 Observe that $q^2\equiv 1\pmod 8$. Now, Theorem~\ref{thrm:tired} implies that $\Gamma_2(\mathrm{PSL}_2(q^m))$ is strongly connected. From this it immediately follows that $\Gamma_2(\Omega_{2m+1}(q))$ is also strongly connected.

We now deal with the second line of Table~\ref{tableOmega}. Suppose $m$ is prime and $q=3$. We have $$\Omega_{2m+1}(q)\ge \Omega_{2m}^+(q)\ge \mathrm{GL}_m(q).2,$$ where $\mathrm{GL}_m(q).2$ is the stabilizer of a direct sum decomposition $V_1\oplus V_2$ of $\mathbb{F}_q^{2m}$. Therefore, we may deduce the strongly connectivity of $\Gamma_2(\Omega_{2m+1}(q))$ from the connectivity of $\Gamma_2(\mathrm{PSL}_m(q))$. 
\end{proof}

\begin{proposition}\label{proposition:omegaplus}
Let $m\ge 4$ be an integer and let $q$ be a prime power. The directed graph $\Gamma_3(\mathrm{P}\Omega_{2m}^+(q))$ is strongly connected.
\end{proposition}
\begin{proof}
Set $G:=\mathrm{P}\Omega_{2m}^+(q)$. As usual, we may only consider the cases when $\Pi(G)$ is disconnected. These cases are reported in Table~\ref{tableOplus}.
\begin{table}[!ht]\centering
\begin{tabular}{cll}
\toprule[1.5pt]
conditions&nr. components & components\\
\midrule[1.5pt]
$m\ge 5$ prime, $q\in \{2,3,5\}$ &2&$\pi\left(q\prod_{i=1}^{m-1}(q^{2i}-1)\right)$, $\pi\left(\frac{q^{m}-1}{q-1}\right)$\\
$m-1$ prime, $q\in \{2,3\}$ , &2&$\pi\left(q(q^{m-1}+1)\prod_{i=1}^{m-2}(q^{2i}-1)\right)$, $\pi\left(\frac{q^{m-1}-1}{\gcd(2,q-1)}\right)$\\
\bottomrule[1.5pt]
\end{tabular}
\caption{Cases when $\mathrm{P}\Omega_{2m}^+(q)$ has disconnected prime graph}\label{tableOplus}
\end{table}

For dealing with the first line of Table~\ref{tableOplus}, we may observe that
$$\Omega_{2m}^+(q)\ge \mathrm{GL}_m(q).2$$
and we may use Proposition~\ref{proposition:PSL}. 
Similarly, for dealing with the second line of Table~\ref{tableOplus}, we may observe that
$$\Omega_{2m}^+(q)\ge  \Omega_{2(m-1)}^+(q)\ge \mathrm{GL}_{m-1}(q).2$$
and we may use Proposition~\ref{proposition:PSL}.
\end{proof}

\begin{proposition}\label{proposition:omegaminus}
Let $m\ge 4$ be an integer and let $q$ be a prime power. The directed graph $\Gamma_2(\mathrm{P}\Omega_{2m}^-(q))$ is strongly connected.
\end{proposition}
\begin{proof}
Set $G:=\mathrm{P}\Omega_{2m}^-(q)$. As usual, we may only consider the cases when $\Pi(G)$ is disconnected. These cases are reported in Table~\ref{tableOmegaminus}.
\begin{table}[!ht]\centering
\begin{tabular}{cll}
\toprule[1.5pt]
conditions&nr. components & components\\
\midrule[1.5pt]
$m=2^\ell\ge 4$, &2&$\pi\left(q\prod_{i=1}^{m-1}(q^{2i}-1)\right)$, $\pi\left(\frac{q^{m}+1}{\gcd(2,q+1)}\right)$\\
$m=2^\ell+1$, &2&$\pi\left(q(q^m+1)\prod_{i=1}^{m-2}(q^{2i}-1)\right)$, $\pi\left(q^{m-1}+1\right)$\\
$\ell\ge 2$, $q=2$,&&\\ 
$m\ge 5$ prime,  &2&$\pi\left(q\prod_{i=1}^{m-1}(q^{2i}-1)\right)$, $\pi\left(\frac{q^{m}+1}{4}\right)$\\
$m\ne 2^\ell+1$, $q=3$&&\\
$m= 2^\ell+1$,  &2&$\pi\left(q(q^m+1)\prod_{i=1}^{m-2}(q^{2i}-1)\right)$, $\pi\left(\frac{q^{m-1}+1}{2}\right)$\\
$\ell\ge 2$, $m$ not prime, $q=3$&&\\
$m= 2^\ell+1$,  &3&$\pi\left(q(q^{m-1}-1)\prod_{i=1}^{m-2}(q^{2i}-1)\right)$, $\pi\left(\frac{q^{m-1}+1}{2}\right)$, $\pi\left(\frac{3^m+1}{4}\right)$\\
$\ell\ge 1$, $m$ prime, $q=3$&&\\
\bottomrule[1.5pt]
\end{tabular}
\caption{Cases when $\mathrm{P}\Omega_{2m}^-(q)$ has disconnected prime graph}\label{tableOmegaminus}
\end{table}

All cases can be deal with by considering various embeddings between classical groups, in a same manner as in Proposition~\ref{proposition:omegaplus}. Indeed, when $m=2^\ell$, by considering the embedding 
$$\Omega_{2m}^-(q)=\Omega_{2^{\ell+1}}^-(q)\ge \Omega_{2^\ell}^-(q^2).2\ge\cdots\ge\Omega_4^-(q^{m/2}).\frac{m}{2}\cong\mathrm{PSL}_2(q^m).\frac{m}{2},$$
we deduce that $\Gamma_2(G)$ is strongly connected from  Theorem~\ref{thrm:tiredtired}.

When $m$ is an odd prime number, by considering the embedding
$$\Omega_{2m}^-(q)\ge \mathrm{GU}_m(q),$$
we deduce that $\Gamma_2(G)$ is strongly connected from Proposition~\ref{proposition:PSUUU}. 

All other cases can be dealt with similarly. 
\end{proof}

\section{Exceptional groups of Lie type}\label{sec:exceptional}
Our argument for dealing with exceptional simple groups of Lie type consists of two parts. In the first part of the argument we make some general considerations and some general computations that hold for each group under consideration. We do this in Section~\ref{sec:exceptionalgeneral}. Then, we specialize, with a case-by-case analysis, the results obtained in Section~\ref{sec:exceptionalgeneral} to study the connectivity of the Engel graph. We do this in Section~\ref{sec:exceptionalspecial}.

\subsection{A general strategy}\label{sec:exceptionalgeneral}
Let $G$ be a finite group and let $H$ be a subgroup of $G$. We let $$1_H^G=\sum_{\chi\in\mathrm{Irr}(G)}a_\chi\chi$$
be the permutation character for the action of $G$ on the right cosets of $H$ in $G$, where we have also decomposed $1_H^G$ as the sum of irreducible complex characters of $G$.

Let $\mathcal{C}$ be a union of conjugacy classes of $G$ and let $C$ be an abelian subgroup of $G$ satisfying:
\begin{description}
\item[Hypothesis~0]$\mathcal{C}\cap H=\emptyset$,
\item[Hypothesis~1]$\mathcal{C}=\{c^g\mid c\in C\setminus\{1\},g\in G\}$,
\item[Hypothesis~2]for each $g\in G\setminus{\bf N}_G(C)$, we have $C\cap C^g=1$,
\item[Hypothesis~3]${\bf N}_G(C)$ is a Frobenius group with kernel $C$.
\end{description}
Examples of this type are quite common for the groups we need to deal with. For instance, when $G={}^2G_2(q)$, we may take $\mathcal{C}$ to be the collection of all elements of $G$ having order a non-identity divisor of $q+\sqrt{3q}+1$ (respectively, $q-\sqrt{3q}+1$) and we may take $C$ to be a maximal non-split torus of order $q+\sqrt{3q}+1$ (respectively, $q-\sqrt{3q}+1$).

We let $$\Delta(H,\mathcal{C})$$ be the auxiliary digraph having vertex set $\mathcal{C}$ and where two vertices $(x,y)\in\mathcal{C}\times\mathcal{C}$ are declared to be adjacent if and only if $yx^{-1}\in H$. Since $H$ is a subgroup of $G$, we deduce that $(x,y)$ is an arc of $\Delta(H,\mathcal{C})$ if and only if so is $(y,x)$. Therefore, $\Delta(H,\mathcal{C})$ is an undirected graph with a loop at each vertex; however, in our opinion, it is more helpful to think of $\Delta(H,\mathcal{C})$ as a directed graph.  We denote with $V\Delta(H,\mathcal{C})$ the vertex set and with $A\Delta(H,\mathcal{C})$ the arc set of $\Delta(H,\mathcal{C})$.

Let $c_1,\ldots,c_\ell$ be a family of representatives for the non-identity ${\bf N}_G(C)$-conjugacy classes of ${\bf N}_G(C)$ contained in $C$. From Hypothesis~3, we have $$\ell=\frac{|C|-1}{|{\bf N}_G(C):C|}.$$ Observe also that, from Hypothesis~2, any two elements of $C$ are $G$-conjugate if and only if they are ${\bf N}_G(C)$-conjugate.  Now, from Hypothesis~1, we deduce that each element in $\mathcal{C}$ has a conjugate in $C$. In particular, we have
\begin{align*}
|\mathcal{C}|&=\sum_{i=1}^\ell|c_i^G|=\sum_{i=1}^\ell|G:{\bf C}_G(c_i)|=\sum_{i=1}^\ell|G:C|=\ell|G:C|=\frac{|G|(|C|-1)}{|{\bf N}_G(C)|}.
\end{align*}
This shows that 
\begin{equation}\label{eq:vertices}
|V\Delta(H,\mathcal{C})|=\frac{|G|(|C|-1)}{|{\bf N}_G(C)|}.
\end{equation}

In what follows, we obtain a formula for computing the cardinality of the arc set $A\Delta(H,\mathcal{C})$. (For simplicity, we let $\Delta:=\Delta(H,\mathcal{C})$.) Clearly,
\begin{equation*}
|A\Delta|=\sum_{h\in H}|A_h|,\,\,\,
\hbox{ where }
A_h=\{(x,y)\in V\Delta\times V\Delta\mid yx^{-1}=h\}.
\end{equation*}
Observe now that $V\Delta=\mathcal{C}$ is a union of conjugacy classes, say $\mathcal{C}=\mathcal{C}_{1}\cup\cdots\cup\mathcal{C}_\ell$, where $\mathcal{C}_i$ is a $G$-conjugacy class for each $i\in \{1,\ldots,\ell\}$. Fix $x_i\in \mathcal{C}_i$. Now,
\begin{align*}
|A_h|&=\sum_{i,j=1}^\ell|\{(x,y)\in\mathcal{C}_i\times\mathcal{C}_j\mid xy^{-1}=h\}|=\sum_{i,j=1}^\ell\frac{|\mathcal{C}_i||\mathcal{C}_j|}{|G|}\sum_{\chi\in\mathrm{Irr}(G)}\frac{\chi(x_i)\overline{\chi(x_j)}\overline{\chi(h)}}{\chi(1)},
\end{align*}
where the second equality follows from~\cite[page~45,~(3.9)]{Isaacs}. Therefore,
\begin{align*}
|A\Delta|&=
\sum_{h\in H}\sum_{i,j=1}^\ell\frac{|\mathcal{C}_i||\mathcal{C}_j|}{|G|}\sum_{\chi\in\mathrm{Irr}(G)}\frac{\chi(x_i)\overline{\chi(x_j)\chi(h)}}{\chi(1)}=
\sum_{i,j=1}^\ell\frac{|\mathcal{C}_i||\mathcal{C}_j|}{|G|}\sum_{\chi\in\mathrm{Irr}(G)}\frac{\chi(x_i)\overline{\chi(x_j)}}{\chi(1)}\sum_{h\in H}\overline{\chi(h)}\\
&=
\sum_{i,j=1}^\ell\frac{|\mathcal{C}_i||\mathcal{C}_j|}{|G|}\sum_{\chi\in\mathrm{Irr}(G)}\frac{\chi(x_i)\overline{\chi(x_j)}}{\chi(1)}|H|\langle\chi, 1_H^G\rangle
=
\sum_{\chi\in\mathrm{Irr}(G)}
\frac{\langle\chi,1_H^G\rangle}{|G:H|\chi(1)}\sum_{i,j=1}^\ell|\mathcal{C}_i||\mathcal{C}_j|\chi(x_i)\overline{\chi(x_j)}
\\
&=
\sum_{\chi\in\mathrm{Irr}(G)}
\frac{\langle\chi,1_H^G\rangle}{|G:H|\chi(1)}\left(\sum_{x\in \mathcal{C}}\chi(x)\right)^2.
\end{align*}
This shows that, to compute the cardinality of $A\Delta$, we need to determine, for each irreducible constituent $\chi$ of $1_H^G$, the quantity $\sum_{x\in\mathcal{C}}\chi(x)$. Using Hypotheses~1,~2 and~3 (in the same manner as in~\eqref{eq:vertices}), we obtain
\begin{align}\label{eq3D428}
\sum_{x\in\mathcal{C}}\chi(x)&=\frac{|G|}{|{\bf N}_G(C)|}\sum_{x\in C\setminus \{1\}}\chi(x)
=\frac{|G|}{|{\bf N}_G(C)|}(|C|\langle\chi_{|C},1\rangle-\chi(1)).
\end{align}
From~\eqref{eq3D428}, we have
\begin{align}\label{eq3D439edges}
|A\Delta(H,\mathcal{C})|&=
\sum_{\chi\in\mathrm{Irr}(G)}
\frac{\langle\chi,1_H^G\rangle}{|G:H|\chi(1)}\left(\frac{|G|}{|{\bf N}_G(C)|}(|C|\langle\chi_{|C},1\rangle-\chi(1))\right)^2\\\nonumber
&=\frac{|G|^2}{|G:H||{\bf N}_G(C)|^2}
\sum_{\chi\in\mathrm{Irr}(G)}
\frac{\langle\chi,1_H^G\rangle}{\chi(1)}\left(|C|\langle\chi_{|C},1\rangle-\chi(1)\right)^2.
\end{align}
It is clear from~\eqref{eq3D428} that, for computing $|A\Delta(H,\mathcal{C})|$, we need to be able to determine, for each constituent $\chi$ of the permutation character $1_H^G$, the multiplicity $\langle\chi,1_H^G\rangle$ of $\chi$ in $1_H^G$ and the multiplicity $\langle\chi_{|C},1\rangle$ of the principal character of $C$ in the restriction $\chi_{|C}$ of $\chi$ to $C$. The following result of Tiep and Zalesski is of great help for computing the latter, when $\chi$ is unipotent and when each element of $C\setminus\{1\}$ is regular.
\begin{lemma}[Lemma~3.2,~\cite{TZ}]\label{TZ}
Let $G$ be a finite group of Lie type. Let $\chi$ be an irreducible unipotent character of
$G$, and $C$ a maximal torus of $G$. Suppose that every element $1 \ne t \in C$ is regular.
Then $\chi$ is constant on $C\setminus \{1\}$ and $\chi(t) \in \{0, 1, -1\}$. Equivalently, 
$\chi_{|C} = 
\frac{\chi(1)-\eta}{|C|}
\cdot 
\rho_C^{\mathrm{reg}} + \eta \cdot 1$,
where $\eta \in \{0, 1, -1\}$ and where $\rho_C^{\mathrm{reg}}$ is the regular representation of $C$.
\end{lemma}

We let
$$c(\Delta(H,\mathcal{C}))$$  be the number of connected components of $\Delta(H,\mathcal{C})$.

For the following lemma we recall the definition of subdegree: given a transitive group $X$ acting on the set $\Omega$ and given $\omega\in \Omega$, the \textbf{\textit{suborbits}} of $X$ are the orbits of the point stabilizer $X_\omega:=\{x\in X\mid \omega^x=\omega\}$. The cardinality of a suborbit is said to be a \textrm{\textbf{subdegree}} of $G$ and, the subdegree $1$ of the suborbit $\{\omega\}$ is said to be the trivial subdegree.  

\begin{lemma}\label{l:components}
The following hold:
\begin{enumerate}
\item\label{components1}For every connected component $X$ of $\Delta(H,\mathcal{C})$ and for every $x\in X$, the subgraph induced by $\Delta(H,\mathcal{C})$ on  $X$ is complete and $X\subseteq Hx$. 
\item\label{components11}We have $c(\Delta(H,\mathcal{C}))= |H\mathcal{C}|/|H|$,  where $H\mathcal{C}:=\{hc\mid h\in H,c\in \mathcal{C}\}$.
 \item\label{components2}We have 
$$c(\Delta(H,\mathcal{C}))\le |G:H|-1$$
and equality holds if and only if $H\mathcal{C}=G\setminus H$.
\item\label{components3}If $c(\Delta(H,\mathcal{C}))<|G:H|-1$, then
$|G:H|-1-c(\Delta(H,\mathcal{C}))\ge d$, where $d$ is the minimum non-trivial subdegree for the permutation action of $G$ on the right cosets of $H$ in $G$.
\end{enumerate}
\end{lemma}
\begin{proof}
Let $X$ be a connected component of $\Delta(H,\mathcal{C})$, let $x\in X$ and let $y,z\in V\Delta(H,\mathcal{C})$ with $x$ adjacent to $y$ and $y$ adjacent to $z$. Then, by definition, we have
$yx^{-1},zy^{-1}\in H$. Since $H$ is a group, we have $zx^{-1}=(zy^{-1})\cdot (yx^{-1})\in H$ and hence $x$ is adjacent to $z$ in $\Delta(H,\mathcal{C})$. Moreover, $z,y\in Hx$. From this it follows that the subgraph induced by $\Delta(H,\mathcal{C})$ on $X$  is complete and that $X\subseteq Hx$. What is more, the only connected component of $\Delta(H,\mathcal{C})$ contained in the coset $Hx$ is $X$. This shows~\eqref{components1}. 

Since the set $H\mathcal{C}$ is a union of right $H$-cosets, this argument also shows that there exists a one to one correspondence between the connected components of $\Delta(H,\mathcal{C})$ and the right cosets of $H$ in $G$ containing some element of $\mathcal{C}$; in particular,~\eqref{components11} holds.

 By Hypothesis~0, $H\cap \mathcal{C}=\emptyset$ and hence $H\cap H\mathcal{C}=\emptyset$. In other words, $H\mathcal{C}\subseteq G\setminus H$. From~\eqref{components11}, we have $c(\Delta(H,\mathcal{C}))\le |G:H|-1$ and  equality holds if and only if $H\mathcal{C}= G\setminus H.$ This proves~\eqref{components2}.

Now, for every $h_1,h_2\in H$ and $c\in \mathcal{C}$, we have $h_1ch_2=h_1h_2c^{h_2}\in H\mathcal{C}$ and hence $H\mathcal{C}H=H\mathcal{C}$. Therefore $H\mathcal{C}$ is a union of $(H,H)$-double cosets. In particular, if $H\mathcal{C}$ is a proper subset of $G\setminus H$, then 
$|(G\setminus H)\setminus H\mathcal{C}|\ge d|H|$, where $d$ is the minimum non-trivial subdegree for the action of $G$ on the right cosets of $H$. This shows~\eqref{components3}.
\end{proof}

Using~\eqref{eq:vertices} and~\eqref{eq3D428}, we are now ready to obtain some numerical important information on the number $c:=c(\Delta(H,\mathcal{C}))$ of connected components of $\Delta(H,\mathcal{C})$. Observe that, if we let $n_1,\ldots,n_c$ be the size of the connected components of $\Delta(H,\mathcal{C})$, then from Lemma~\ref{l:components}~\eqref{components1} we have
\begin{align*}
|V\Delta(H,\mathcal{C})|&=\sum_{i=1}^{c}n_i,\quad|A\Delta(H,\mathcal{C})|=\sum_{i=1}^{c}n_i^2,
\end{align*}
because each connected component is a complete graph. In particular,
using the Cauchy-Swartz inequality, we have 
$$c\ge \frac{(\sum_{i=1}^cn_i)^2}{\sum_{i=1}^cn_i^2}=\frac{|V\Delta(H,\mathcal{C})|^2}{|A\Delta(H,\mathcal{C})|}.$$
Now, from~\eqref{eq:vertices} and~\eqref{eq3D439edges}, we obtain
\begin{equation}\label{eureka}
c(\Delta(H,\mathcal{C}))
\geq |G:H|\cdot \frac{(|C|-1)^2}{
\sum_{\chi\in\mathrm{Irr}(G)}\frac{\langle\chi,1_H^G\rangle}{\chi(1)}(|C|\langle\chi_{|C},1 \rangle-\chi(1))^2}.
\end{equation}

\subsection{Application of the  general strategy}\label{sec:exceptionalspecial}
In this section, we apply the results from Section~\ref{sec:exceptionalgeneral} to deduce some information on the connectivity of the Engel graph for simple exceptional groups of Lie type.

\subsubsection{The Steinberg triality group}\label{sec:3D4}
Let $q$ be a prime power and let  $$G:={}^3D_4(q).$$ For information on the subgroup structure of $G$, we use~\cite{Kleidman}. See also~\cite[Table~8.51]{bray}. Now, $G$ has two maximal parabolic subgroups. Here, we consider a maximal parabolic subgroup $H$ with 
\begin{equation}
H\cong E_q^{1+8}:((q-1)\circ\mathrm{SL}_2(q^3)).\gcd(q-1,2),
\end{equation}
see for instance~\cite{Kleidman}. In the action of $G={}^3D_4(q)$ on the associated generalized Hexagon of order $(q,q^3)$, $H$ is a vertex stabilizer.

\begin{lemma}\label{3D4lemma0}
The subdegrees for the primitive action of ${}^3D_4(q)$ on the right cosets of $H$ are 
$1$, $q(q^3+1)$, $q^9$ and $q^5(q^3+1)$. In particular, the minimum non-trivial subdegree is $q(q^3+1)$.
\end{lemma}
\begin{proof}
This follows from~\cite[Theorem~3, page~13]{Vasilev}.
\end{proof}

We let $$\mathcal{C}:=\{g\in G\mid {\bf o}(g)>1, {\bf o}(g)\hbox{ divides } q^4-q^2+1\}$$
and we let $C$ be a maximal torus of $G$ of order $q^4-q^2+1$.
Using the subgroup structure of $G={}^3D_4(q)$ (see~\cite{Kleidman}), it is not hard to verify that Hypotheses~0--4 are satisfied with respect to this choice of $H$, $\mathcal{C}$ and $C$. We prove the following strong result.

\begin{lemma}\label{3D4lemma1}
We have $H\mathcal{C}=G\setminus H$.
\end{lemma}
\begin{proof}
To prove this result, we use Lemma~\ref{l:components} and~\eqref{eureka}.

For some information on the characters of $G={}^3D_4(q)$, we refer to~\cite{Spalstein}. Using~\cite{Spalstein} and the notation therein, we have
\begin{equation}
\label{3D4eq:1}
1_H^G=1+\varepsilon_1+\rho_1+\rho_2,
\end{equation}
where 
$$1(1)=1,\,\varepsilon_1(1)=q(q^4-q^2+1),\,\rho_1(1)=\frac{1}{2}q^3(q^3+1)^2,\,\rho_2(1)=\frac{1}{2}q^3(q+1)^2(q^4-q^2+1).$$
In particular, we deduce
\begin{equation*}
\langle 1_H^G,1_H^G\rangle=
1+1+1+1=4 
\end{equation*}
and hence the rank of the action of $G$ on the right cosets of $H$ is $4$. Incidentally,  the fact that $G$ has rank $4$ in its action on the right cosets of $H$ follows also from Lemma~\ref{3D4lemma0}, where we have listed all subdegrees.

Let $\chi$ be an irreducible constituent of $1_H^G$. The quantity $|C|\langle\chi_{|C},1\rangle-\chi(1)$ can be computed using the work of Spaltenstein~\cite{Spalstein} on the irreducible characters of ${}^3D_4(q)$. We have
\begin{align}
\label{eq3D429}
|C|\langle\chi_{|C},1\rangle-\chi(1)=
\begin{cases}
q^4-q^2=|C|-1&\textrm{when }\chi=1,\\
0&\textrm{when }\chi=\varepsilon_1,\\
-(q^4-q^2)=-|C|+1&\textrm{when }\chi=\rho_1,\\
0&\textrm{when }\chi=\rho_2.
\end{cases}
\end{align}
Observe that, since each irreducible constituent of $1_H^G$ is unipotent and since each non-identity element of $C$ is regular,~\eqref{eq3D429} is perfectly in line with Lemma~\ref{TZ}. From~\eqref{eq3D429} and with a computation, we find
\begin{align}\label{eq3D439}
\sum_{
\chi\in \mathrm{Irr}(G)
}
\frac{\langle\chi,1_H^G\rangle}{\chi(1)}
\left(|C|\langle\chi_{|C},1\rangle-\chi(1)\right)^2&=(|C|-1)^2+\frac{(|C|-1)^2}{q^3(q^3+1)^2/2}\\\nonumber
&=(|C|-1)^2f(q),
\end{align} 
where for simplicity $f(q):=1+1/(q^3(q^3+1)^2/2)$.

Summing up, from~\eqref{eureka} and~\eqref{eq3D439}, we obtain
\begin{equation*}
c(\Delta(H,\mathcal{C}))
\ge
\frac{|G:H|}{f(q)}> |G:H|-1-q(q^3+1),
\end{equation*}
where the last inequality follows with an elementary computation using the function $f(q)$ and using the fact that $|G:H|=(q+1)(q^8+q^4+1)$.
By Lemma~\ref{l:components}~\eqref{components3} and Lemma~\ref{3D4lemma0}, we deduce $c(\Delta(H,\mathcal{C}))=|G:H|-1$.
\end{proof}

\begin{lemma}\label{l:engel3D4}
There exists a non-identity element $g\in G={}^3D_4(q)$ having order a divisor of $q^4-q^2+1$ and a non-identity unipotent element $\iota$ with  $[g,\iota,\iota]=1$.  
\end{lemma}
\begin{proof}
We may write $H=L_H\rtimes U_H$, where $L_H$ is the Levi complement of the parabolic subgroup $H$ and $U_H$ is the unipotent radical. We now collect some information on $U_H$, see~\cite{Geck,himstedt}. From~\cite[page~785]{himstedt}, we have $$U_H=X_{\beta} X_{\alpha+\beta} X_{2\alpha+\beta} X_{3\alpha+\beta}X_{3\alpha+2\beta},$$ where each ``$X$'' is a root subgroup with respect to a positive root for the root system $G_2$. 
From~\cite[page~786]{himstedt} and the notation therein, $H={\bf N}_G(X_{3\alpha+2\beta})$;
 moreover, $X_{3\alpha+2\beta}$ is the center, the derived subgroup and the Frattini subgroup of $U_H$. Now, from~\cite[Table~2.4]{Geck}, we see that there exists $n\in G$ with 
 \begin{align}\label{eq:3d4e}X_{3\alpha+2\beta}^n=X_{3\alpha+\beta}.
 \end{align} (Incidentally, the element $n$ in~\cite[Table~2.4]{Geck} is the element labeled $n_\beta$ in the Weyl group.)

As $X_{3\alpha+2\beta}^H=X_{3\alpha+2\beta}$, from Lemma~\ref{3D4lemma1}, we have
\begin{align}\label{eq:3d4ee}
X_{3\alpha+2\beta}^{\mathcal{C}}
=X_{3\alpha+2\beta}^{H\mathcal{C}}=
X_{3\alpha+2\beta}^{G\setminus H}=
\{X_{3\alpha+2\beta}^g\mid g\in G\}\setminus\{X_{3\alpha+2\beta}\}.
\end{align}

Now, from~\eqref{eq:3d4e} and~\eqref{eq:3d4ee}, there exists $g\in \mathcal{C}$ with $X_{3\alpha+2\beta}^g=X_{3\alpha+\beta}$. Let $\iota\in X_{3\alpha+2\beta}$ with $\iota\ne 1$. Then $\iota^g\in X_{3\alpha+\beta}\le U_H$ and $[\iota,g]=\iota^{-1}\iota^g\in U_H$. Since $\iota\in X_{3\alpha+2\beta}\le {\bf Z}(U_H)$, we deduce that 
$\iota $ commutes with $[\iota,g]$, that is, $[g,\iota,\iota]=1$.
\end{proof}

\begin{proposition}
The directed graph $\Gamma_2({}{}^3D_4(q))$ is strongly connected.
\end{proposition}
\begin{proof}
Let $G:={}^3D_4(q)$. 
By Theorem~\ref{theorem111}, all elements of even order of $G$ are in the same connected component of $\Gamma_1(G)$. In particular, we use the prime graph of $G$ to deduce properties of the commuting graph of $G$, see Corollary~\ref{cor}. From~\cite[Table~1]{KoMa}, the connected components of the prime graph of $G={}^3D_4(q)$ are 
$$\pi(q(q^6-1))\,\,\,\hbox{ and }\,\,\,\pi(q^4-q^2+1).$$

Lemma~\ref{l:engel3D4} guarantees that, there exists an element $g$ having order divisible by an element in the second connected component of $\Pi(G)$ and an element $z$ having order divisible by an element in the first connected component of $\Pi(G)$ 
with $g\mapsto_2 z$. 

Let $g$ be an element of $G$ with ${\bf o}(g)=q^4-q^2+1$ and let $C:=\langle g\rangle$. From~\cite{Kleidman}, we have ${\bf N}_G(C)=C\rtimes\langle z\rangle$, where ${\bf o}(z)=4$. In particular, $[z,g,g]=1$ and $z\mapsto_2 g$. This shows that there exists an element $g$ having order divisible by an element in the second connected component of $\Pi(G)$ and an element $z$ having order divisible by an element in the first connected component of $\Pi(G)$ 
$z$ with $z\mapsto_2 g$. Therefore, $\Gamma_2(G)$ is strongly connected.
\end{proof}

\subsubsection{The Ree groups of type ${}^2G_2(q)$}\label{sec:2G2}
Let $q$ be a prime power with $q:=3^{2k+1}$, for some positive integer $k\ge 1$, and let $$G:={}^2G_2(q).$$ For information on the subgroup structure and on the conjugacy classes of $G$, we use~\cite[Table~8.43]{bray} and~\cite{Ree}.

Observe that $G$ has a unique conjugacy class of involutions. Let $\iota$ be an involution of $G:={}^2G_2(q)$ and let 
\begin{equation}
H:={\bf C}_G(\iota)\cong C_2\times \mathrm{PSL}_2(q).
\end{equation}
Now, $H$ is a maximal subgroup of $G$. 

\begin{lemma}[Theorem~$1.5$, \cite{HHP}]\label{2G2lemma00}
The subdegrees for the primitive action of $G={}^2G_2(q)$ on the right cosets of $H$ are (not counted with multiplicity)
$1$, $q(q-1)/2$, $(q-1)(q+1)$, $(q-1)q(q+1)/4$, $q(q-1)(q+1)/2$ and $q(q-1)(q+1)$ . In particular, the minimum non-trivial subdegree is $q(q-1)/2$.
\end{lemma}

Given $\varepsilon\in \{1,-1\}$, 
let $$\mathcal{C}_\varepsilon:=\{g\in G\mid {\bf o}(g)>1, {\bf o}(g)\hbox{ divides } q+\varepsilon\sqrt{3q}+1\}$$
and let $C_\varepsilon$ be a maximal torus of $G$ of order $q+\varepsilon\sqrt{2q}+1$. It is not hard to verify that Hypotheses~0--4 are satisfied with respect to this choice of $H$, $\mathcal{C}_\varepsilon$ and $C_\varepsilon$. We prove the following strong result.
\begin{lemma}\label{2G2lemma1}
We have $H\mathcal{C}_\varepsilon=G\setminus H$, where $H\mathcal{C}_\varepsilon:=\{hx\mid h\in H,x\in \mathcal{C}_\varepsilon\}$.
\end{lemma}
\begin{proof}
We use the character table of $G={}^2G_2(q)$, see~\cite{Ward}. As usual $1_H^G$ is the permutation character for the primitive action of $G$ on the right cosets of $H$. Using~\cite{Ward} and the notation therein, we have
\begin{equation}
\label{2G2eq:1}
1_H^G=\xi_1+2\xi_3+2\sum_{r=1}^{\frac{q-3}{4}}\eta_r+
4\sum_{t=1}^{\frac{q-3}{24}}\eta_t+
\sum_{i=1}^{
\frac{q+\sqrt{3q}}{6}
}\eta_i^+
+\sum_{i=1}^{
\frac{q-\sqrt{3q}}{6}
}\eta_i^-
\end{equation}
and
\begin{align*}
\xi_1(1)&=1,\,\xi_3(q)=q^3,\,\eta_r(1)=q^3+1,\,\eta_t(1)=(q-1)(q^2-q+1),\\
\eta_i^+(1)&=(q^2-1)(q+1-\sqrt{3q}),\,\eta_i^-(1)=(q^2-1)(q+1+\sqrt{3q}).
\end{align*}
In particular, we deduce
\begin{equation*}
\langle 1_H^G,1_H^G\rangle=
1+
4+
4\cdot\frac{q-3}{4}+ 
16\cdot\frac{q-3}{24}+
1\cdot \frac{q+\sqrt{3q}}{6}+
1\cdot \frac{q-\sqrt{3q}}{6}=2q  
\end{equation*}
and hence the rank of the action of $G$ on the right cosets of $H$ is $2q$.

Let $\chi$ be an irreducible constituent of $1_H^G$.
The  quantity $|C_\varepsilon|\langle\chi_{|C_\varepsilon},1\rangle-\chi(1)$ can be computed using the work of Ward~\cite{Ward} on the irreducible characters of ${}^2G_2(q)$. We have
\begin{align}\label{eq2G21}
|C_\varepsilon|\langle\chi_{|C_\varepsilon},1\rangle-\chi(1)=
\begin{cases}
|C_\varepsilon|-1&\textrm{when }\chi=\xi_1,\\
-|C_\varepsilon|+1&\textrm{when }\chi=\xi_3,\\
6&\textrm{when }\chi=\eta_i^\varepsilon\\
0&\textrm{otherwise}.
\end{cases}
\end{align}
Again, observe that,~\eqref{eq2G21} is  perfectly in line with Lemma~\ref{TZ}. Indeed, each non-identity element of $C_\varepsilon$ is regular and the constituents $\xi_1$, $\xi_3$, $\eta_r$ and $\eta_t$ are unipotent, whereas $\eta_i^\pm$ are not unipotent. 
From~\eqref{eq2G21} and with a computation, we find
\begin{align}\label{eq2G2111}
\sum_{\chi\in \mathrm{Irr}(G)}\frac{\langle\chi,1_H^G\rangle}{\chi(1)}\left(|C_\varepsilon|\langle\chi_{|C_\varepsilon}-\chi(1)\rangle\right)^2&=(|C_\varepsilon|-1)^2\left(1+\frac{2}{q^3}+\frac{6}{(q^2-1)(q+1-\varepsilon\sqrt{3q})(q+\varepsilon\sqrt{3q})}\right)\\\nonumber
&=(|C_\varepsilon|-1)^2f_\varepsilon(q),
\end{align} 
where for simplicity, we denote by $f_\varepsilon(q)$ the ugly factor appearing on the right hand side of~\eqref{eq2G2111}.

Summing up, from~\eqref{eureka} and~\eqref{eq2G21}, we obtain
\begin{equation*}
c(\Delta(H,\mathcal{C}))
\ge
\frac{|G:H|}{f_\varepsilon(q)}> |G:H|-1-q(q-1)/2,
\end{equation*}
where the last inequality follows with an elementary computation using the function $f_\varepsilon(q)$ and using the fact that $|G:H|=q^2(q^2-q+1)$.
By Lemma~\ref{l:components}~\eqref{components3} and Lemma~\ref{2G2lemma00}, we have $c(\Delta(H,\mathcal{C}))=|G:H|-1$.
\end{proof}

We are now ready to deduce some information on the Engel graph of $G={}^2G_2(q)$.

\begin{lemma}\label{l:engel2G2}
For each $\varepsilon\in \{1,-1\}$, there exists a non-identity element $g_\varepsilon\in G={}^2G_2(q)$ having order a divisor of $q+\varepsilon\sqrt{3q}+1$ with  $[g_\varepsilon,\iota,\iota]=1$. (Recall that $\iota$ is an involution of $G$.) 
\end{lemma}
\begin{proof}
Let $\mathcal{I}$ be the collection of all the involutions of $G$. Since $G$ has a unique conjugacy class of involutions, we have $\mathcal{I}=\iota^G$. Now let $\tau\in\mathcal{I}$ with $\langle\iota,\tau\rangle$ an elementary abelian $2$-group and observe that  this is possible because a Sylow $2$-subgroup of $G$ is elementary abelian or order $8$. 
Now, from Lemma~\ref{2G2lemma1}, since $H={\bf C}_G(\iota)$, we have
$$
\iota^{\mathcal{C}_\varepsilon}
=\iota^{HC_\varepsilon}=
\iota^{G\setminus H}=
\iota^G\setminus\{\iota\}=\mathcal{I}\setminus\{\iota\}.
$$
In particular, for each $\varepsilon\in\{1,-1\}$, there exists $g_\varepsilon\in \mathcal{C}_\varepsilon$ with $\iota^{g_\varepsilon}=\tau$. Therefore, we have
$$[\iota,g_\varepsilon]=\iota \iota^{g_\varepsilon}=\iota\tau$$
 and hence $[g_\varepsilon,\iota,\iota]=1$.
\end{proof}

\begin{proposition}
The directed graph $\Gamma_2({}{}^2G_2(q))$ is strongly connected.
\end{proposition}
\begin{proof}
Let $G:={}^2G_2(q)$. Observe that in this proof we may not use Theorem~\ref{theorem111} and Corollary~\ref{cor}, because the Ree group ${}^2G_2(q)$ is an exception in Theorem~\ref{theorem111} and Corollary~\ref{cor}.

Let $\iota$ be an involution of $G$ and let $\mathcal{B}$ be the conjugacy class consisting of involutions. We claim that $\mathcal{B}$ is contained in a unique strongly connected component of $\Gamma_2(G)$. To prove this we use some geometric properties of $G$, see~\cite{Luneburg}. 

Let $\mathcal{P}$ be the collection of all Sylow $3$-subgroups. We consider the incidence relation $\mathcal{I}\subseteq \mathcal{P}\times\mathcal{B}$, where $(P,\sigma)\in\mathcal{I}$ if and only if $\sigma\in {\bf N}_G(P)$. From~\cite{Luneburg}, we see that $(\mathcal{P},\mathcal{B})$ is a $2$-$(q^3+1,q+1,1)$ design. This means that $\mathcal{P}$ has cardinality $q^3+1$, each element of $\mathcal{B}$ is incident to $q+1$ elements of $\mathcal{P}$ and, given any two distinct elements of $\mathcal{P}$, there exists a unique element of $\mathcal{B}$ incident with both. 

Observe that ${\bf N}_G(P)\cong E_{q}^{1+1+1}:(q-1)$ for each element $P\in \mathcal{P}$. From this, it is not hard to verify that the subgraph induced by $\Gamma_2(G)$ on ${\bf N}_G(P)$ is strongly connected. Now, let $\iota,\iota'\in\mathcal{B}$ and let $b_\iota:=\{P\in\mathcal{P}\mid (P,\iota)\in\mathcal{I}\}$ and $b_{\iota'}:=\{P\in\mathcal{P}\mid (P,\iota')\in\mathcal{I}\}$. In particular, $b_\iota$ and $b_{\iota'}$ are two sets of cardinality $q+1$.  Now, let $P\in b_\iota$ and $P'\in b_{\iota'}$. Now, $\iota\in {\bf N}_G(P)$ and $\iota'\in{\bf N}_G(P')$ and ${\bf N}_G(P)\cap {\bf N}_G(P')$ is cyclic of order $q-1$ (actually, when $P=P'$, we have ${\bf N}_G(P)={\bf N}_G(P')$, but this degenerate case is of no concern to us). Therefore, we may reach $\iota'$ from $\iota$ in $\Gamma_2(G)$ using the strongly connectivity of the subgraph induced on ${\bf N}_G(P)$ and on ${\bf N}_G(P')$. Thus $\iota\mapsto_2\iota'$ and $\iota'\mapsto_2\iota$.

From the previous paragraphs it immediately follows that all elements of even order of $G$ are in the same strongly connected component of $\Gamma_2(G)$. Now, using the subgroup structure of $G$, see~\cite[Table~8.43]{bray}, to conclude that $\Gamma_2(G)$ is strongly connected it suffices to prove that, for each $\varepsilon\in \{1,-1\}$, there exists a non-identity element $g$ having order a divisor of $q+\varepsilon\sqrt{3q}+1$ such that $g\mapsto_2\iota$. But this is exactly Lemma~\ref{l:engel2G2}.
\end{proof}

\subsubsection{The Ree groups of type ${}^2F_4(q)$}\label{sec:2F4}
Let $q$ be a prime power with  $q:=2^{2k+1}$, for some positive integer $k\ge 1$. (The strong connectivity of the Engel graph of Tit's group ${}^2F_4(2)'$ can be checked with the auxiliary help of a computer. In fact, $\Gamma_2({}^2F_4(2)')$ is strongly connected.) 
Recall that $${}^2F_4(q)=q^{12}(q^6+1)(q^4-1)(q^3+1)(q-1).$$

The maximal subgroups of ${}^2F_4(q)$ have been classified by Malle~\cite{Malle}. We let
\begin{equation}\label{structure:H}
H\cong q^{10}:({}^2B_2(q)\times (q-1))
\end{equation}
be one of the maximal parabolic subgroups of ${}^2F_4(q)$.
\begin{lemma}[Theorem~5, \cite{Vasilev}]\label{2F4lemma0}
When $k>0$, the subdegrees for the primitive action of ${}^2F_4(q)$ on the right cosets of $H$ are 
$1$, $q(q^2+1)$, $q^{10}$, $q^4(q^2+1)$ and $q^7(q^2+1)$. In particular, the minimum non-trivial subdegree is $q(q^2+1)$.
\end{lemma}
\begin{proof}
This follows from~\cite[Theorem~5, page~19]{Vasilev}.
\end{proof}

Given $\varepsilon\in \{1,-1\}$, we let $$\mathcal{C}_\varepsilon:=\{g\in G\mid {\bf o}(g)>1, {\bf o}(g)\hbox{ divides } q^2+\varepsilon\sqrt{2q^3}+q+\varepsilon\sqrt{2q}+1\}$$
and we let $C_\varepsilon$ be a maximal torus of $G$ of order $q^2+\sqrt{2q^3}+q^2+\varepsilon\sqrt{2q}+1$. 
It is not hard to verify that Hypotheses~0--4 are satisfied with
respect to this choice of $H$, $\mathcal{C}_\varepsilon$  and $C_\varepsilon$.
We prove the following strong result.
\begin{lemma}\label{2F4lemma1}
For $q>2$, we have $H\mathcal{C}_\varepsilon=G\setminus H$, where $H\mathcal{C}_\varepsilon:=\{hx\mid h\in H,x\in \mathcal{C}_\varepsilon\}$.
\end{lemma}
\begin{proof}
We use the unipotent characters of $G={}^2F_4(q)$, see~\cite{Mallec}. Let $1_H^G$ be the permutation character for the primitive action of $G$ on the right cosets of $H$. Using~\cite{Mallec} and the notation therein, we have
\begin{equation}
\label{2F4eq:1}
1_H^G=\chi_1+\chi_2+\chi_9+\chi_{10}+\chi_{11} 
\end{equation}
and 
\begin{align*}\chi_1(1)&=1,\,\chi_2(1)=q(q^2-q+1)(q^4-q^2+1),\\
\chi_9(1)&=\frac{1}{4}q^2(q+1)^2(q-\sqrt{2q}+1)^2(q^2-q+1)(q^2+\sqrt{2q^3}+q+\sqrt{2q}+1),\\
\chi_{10}(1)&=\frac{1}{4}q^2(q+1)^2(q+\sqrt{2q}+1)^2(q^2-q+1)(q^2-\sqrt{2q^3}+q-\sqrt{2q}+1),\\
\chi_{11}(1)&=\frac{1}{2}q^2(q^2+1)^2(q^4-q^2+1).
\end{align*}
In particular, the rank of the action of $G$ on the right cosets of $H$ is $5$.

Let $\chi$ be an irreducible constituent of $1_H^G$. The quantity $|C_\varepsilon|\langle\chi_{|C_\varepsilon},1\rangle-\chi(1)$ can be computed using the work of Malle~\cite[Tabelle~8]{Mallec} on the unipotent irreducible characters of ${}^2F_4(q)$. We have
\begin{align}\label{eq2F441}
|C_\varepsilon|\langle\chi_{|C_\varepsilon},1\rangle-\chi(1)=
\begin{cases}
|C_\varepsilon|-1&\textrm{when }\chi=\chi_1,\\
-|C_\varepsilon|+1&\textrm{when }\chi=\chi_9 \textrm{ and }\varepsilon=-1,\\
-|C_\varepsilon|+1&\textrm{when }\chi=\chi_{10} \textrm{ and }\varepsilon=1,\\
0&\textrm{otherwise}.
\end{cases}
\end{align}
As usual, observe that,~\eqref{eq2F441} is perfectly in line with Lemma~\ref{TZ}. Indeed, each non-identity element of $C_\varepsilon$ is regular and the
constituents of $1_H^G$ are unipotent, because $H$ is parabolic.

From~\eqref{eq2F441} and a computation, we find
\begin{align}\label{eq2F4111}
\sum_{\chi\in \mathrm{Irr}(G)}\frac{\langle\chi,1_H^G\rangle}{\chi(1)}\left(|C_\varepsilon|\langle \chi_{|C_\varepsilon},1\rangle-\chi(1)\right)^2&=(|C_\varepsilon|-1)^2\left(1+\frac{4}{q^2(q+1)^2(q^2-q+1)(q+\varepsilon\sqrt{2q}+1)^2(q^2-\varepsilon\sqrt{2q^3}+q-\varepsilon\sqrt{2q}+1)}\right)\\\nonumber
&=(|C_\varepsilon|-1)^2f_\varepsilon(q),
\end{align} 
where for simplicity, we denote by $f_\varepsilon(q)$ the ugly factor appearing on the right hand side of~\eqref{eq2F4111}.

Summing up, from~\eqref{eureka} and~\eqref{eq2F441}, we obtain
\begin{equation*}
c(\Delta(H,\mathcal{C}))
\ge
\frac{|G:H|}{f_\varepsilon(q)}> |G:H|-1-q(q^2+1),
\end{equation*}
where the last inequality follows with an elementary computation using the function $f_\varepsilon(q)$ and using the fact that $|G:H|=(q^6+1)(q^3+1)(q+1)$.
By Lemma~\ref{l:components}~\eqref{components3} and Lemma~\ref{2F4lemma0}, we have $c(\Delta(H,\mathcal{C}))=|G:H|-1$.

\end{proof}

\begin{lemma}\label{l:engel2F4}
Given $\varepsilon\in \{1,-1\}$, when $q>2$, there exists a non-identity element $g_\varepsilon\in G={}^2F_4(q)$ having order a divisor of $q^2+\varepsilon\sqrt{2q^3}+q+\varepsilon\sqrt{2q}+1$ and a non-identity unipotent element $\iota$ with  $[g_\varepsilon,\iota,\iota]=1$.  
\end{lemma}
\begin{proof}
We may write $H=L_H\rtimes U_H$, where $L_H\cong {}^2B_2(q)\times (q-1)$ is the Levi complement of the parabolic subgroup $H$ and $U_H$ is the unipotent radical. We now collect some information on $U_H$, see~\cite{Ree2,Shinoda}. Let $U$ be a Sylow $2$-subgroup of $H$. From~\cite{Ree2,Shinoda,Shinoda1} and from the notation wherein, we have $$U=
\{\alpha_1(t)\mid t\in \mathbb{F}_q\}\cdot
\{\alpha_2(t)\mid t\in \mathbb{F}_q\}\cdots
\{\alpha_{12}(t)\mid t\in \mathbb{F}_q\}.$$
 From~\cite[Section~4]{Ree2}, the center of $U$ has order $q$ and consists of the root subgroup $X:=\{\alpha_{12}(t)\mid t\in\mathbb{F}_q\}$. (The element $\alpha_{12}(t)$ is defined in~\cite[page~407]{Ree2}.)  
 
 From~\cite[Table~3]{himstedt1} see also~\cite[(2.2)]{Shinoda1}, we see that there exists an element $n_a$ in the Weyl group of $G$ with 
 $\{\alpha_{12}(t)\mid t\in \mathbb{F}_q\}^{n_a}=\{\alpha_{11}(t)\mid t\in \mathbb{F}_q\}$.
  Set $X_{12}:=\{\alpha_{12}(t)\mid t\in \mathbb{F}_q\}$ 
  and 
  $X_{11}:=\{\alpha_{11}(t)\mid t\in \mathbb{F}_q\}$.

As $X_{12}^H=X_{12}$ because $H$ normalizes $X_{12}$, from Lemma~\ref{2F4lemma1}, we have
\begin{align}\label{eq:2F42F4}
X_{12}^{\mathcal{C}_\varepsilon}
=X_{12}^{H\mathcal{C}_\varepsilon}=
X_{12}^{G\setminus H}=
\{X_{12}^g\mid g\in G\}\setminus\{X_{12}\}.
\end{align}

Now, from~\eqref{eq:2F42F4}, there exists $c\in \mathcal{C}_\varepsilon$ with $\alpha_{12}(t)^c\in X_{11}\le U_H$. Let $\iota=\alpha_{12}(t)$. Then $\iota^c\in U_H$ and $[\iota,c]=\iota^{-1}\iota^c\in U_H$. Since $\iota\in {\bf Z}(U_H)$, we deduce that 
$\iota $ commutes with $[\iota,c]$, that is, $[c,\iota,\iota]=1$.
\end{proof}

\begin{proposition}
The directed graph $\Gamma_2({}{}^2F_4(q))$ is strongly connected.
\end{proposition}
\begin{proof}
Let $G:={}^2F_4(q)$. From~\eqref{structure:H}, it is not hard to verify that the subgraph induced by $\Gamma_2(G)$ on $H$ is strongly connected. Moreover, from Lemma~\ref{2F4lemma0}, we have $H\cap H^g\ne 1$, for every $g\in G$.
 From these two facts, it follows that all elements of even order of $G$ are in the same strongly connected component of $\Gamma_2(G)$. 

 Moreover, if we let $C=\langle g\rangle$ be a maximal torus of $G$ with $|C|=c$, then $C<{\bf N}_G(C)$ and $|{\bf N}_G(C):C|$ is divisible by $2$. Therefore, there exists an involution $\iota\in {\bf N}_G(C)$. This gives $\iota\mapsto_2 g$. Now, using the subgroup structure of $G$, to conclude that $\Gamma_2(G)$ is strongly connected it suffices to prove that, for each $\varepsilon\in \{1,-1\}$, there exists a non-identity element $g$ having order a divisor of $q^2+\varepsilon\sqrt{2q^3}+q+\varepsilon\sqrt{2q}+1$ and an element $\iota$ having order divisible by a prime in the first connected component of $\Pi(G)$ such that $g\mapsto_2\iota$. This information is obtained in Lemma~\ref{l:engel2F4}.
\end{proof}

\subsubsection{The exceptional group $E_8(q)$}\label{sec:E8}
 Let $q$ be a prime power and let $G := E_8 (q)$. For information on the subgroup structure
of $G$, we use~\cite{KZ}. Now, $G$ has eight maximal parabolic subgroups. Here, we consider a maximal parabolic subgroup $H$ with
\begin{align}
|G:H|=\frac{(q^{30}-1)(q^{12}+1)(q^{10}+1)(q^6+1)}{q-1}. 
\end{align}

\begin{lemma}\label{E8:1} The subdegrees for the primitive action of $G$ on the right cosets of $H$ are 
\begin{align*}
&1,\,\,\,\, \frac{q(q^{14} -1)(q^9+1)(q^5+1)}{q-1},\,\,\, q^{57},\\
&q^{29}\cdot\frac{(q^{14}-1)(q^9+1)(q^5+1)}{q-1},\,\,\,q^{12}\cdot \frac{(q^{14}-1)(q^{12}+q^6+1)(q^8+q^4+1)}{q-1}.\nonumber
\end{align*}
In particular, the minimum non-trivial subdegree is $q\cdot\frac{(q^{14}-1)(q^9+1)(q^5+1)}{(q-1)}$.
\end{lemma}
\begin{proof} This follows from~\cite[Theorem~3]{Vasilev1}.
\end{proof}

Let $1_H^G$ be the permutation character for the action of $G$ on the right cosets of $H$. From Lemma~\ref{E8:1}, $1_H^G$ is the sum of five irreducible constituents. Now, the list of the complex character degrees of $G=E_8(q)$ can be found in the homepage of Frank L\"{u}beck~\cite{Frank}. (The data is readily available, but this information can also be deduce from~\cite{carter}.) We have
$$1_H^G=\chi_1+\chi_2+\chi_3+\chi_4+\chi_5,$$
where
\begin{align*}
\chi_1(1)=&1,\\
\chi_2(1)=&q(q^4-q^2+1)(q^8-q^6+q^4-q^2+1)(q^4+1)(q^8-q^4+1)(q^2+1)^2,\\
\chi_3(1)=&q^2(q^8-q^6+q^4-q^2+1)(q^4+q^3+q^2+q+1)(q^4-q^3+q^2-q+1)(q^6+q^5+q^4+q^3+q^2+q+1)\cdot \\
&\cdot(q^6-q^5+q^4-q^3+q^2-q+1)
(q^8-q^7+q^5-q^4+q^3-q+1)(q^8+q^7-q^5-q^4-q^3+q+1),\\
\chi_4(1)=&\frac{1}{2}q^3(q^6+q^3+1)(q^4+1)(q^4-q^2+1)(q^6-q^5+q^4-q^3+q^2-q+1)(q^6+q^5+q^4+q^3+q^2+q+1)\cdot \\
&\cdot(q^8-q^6+q^4-q^2+1)(q^8+q^7-q^5-q^4-q^3+q+1)(q^8-q^4+1)(q^2+1)^2,\\
\chi_5(1)&=\frac{1}{2}q^3(q^6-q^3+1)(q^4+1)(q^6-q^5+q^4-q^3+q^2-q+1)(q^6+q^5+q^4+q^3+q^2+q+1)\cdot\\
&\cdot (q^8+q^7-q^5-q^4-q^3+q+1)(q^8-q^4+1)(q^2-q+1)^2(q^4-q^3+q^2-q+1)^2(q+1)^4.
\end{align*}

Now, there are two cases that we want to consider, and these cases depend on $q$ modulo $5$.
\begin{description}
\item[Case~A] When $q\equiv 2,3\pmod 5$, we let $$c\in \left\{\frac{q^{10}-q^5+1}{q^2-q+1},\frac{q^{10}+q^5+1}{q^2+q+1},q^8-q^4+1\right\},$$
we let $C$ be a maximal torus of $G$ of order $c$ and we let $\mathcal{C} := \{g \in G | {\bf o}(g) > 1, {\bf o}(g) \mid  c\}$,
\item[Case~B] When $q\equiv 0,1,4\pmod 5$, we let $$c\in \left\{\frac{q^{10}-q^5+1}{q^2-q+1},\frac{q^{10}+q^5+1}{q^2+q+1},q^8-q^4+1,\frac{q^{10}+1}{q^2+1}\right\},$$
we let $C$ be a maximal torus of $G$ of order $c$ and we let $\mathcal{C} := \{g \in G | {\bf o}(g) > 1, {\bf o}(g) \hbox{ divides }  c\}$.
\end{description}
Considering the various possibilities for $q\pmod 5$ and considering the maximal tori of $G$ in~\cite{KZ}, it is not hard to verify that Hypotheses 0--4 are satisfied with
respect to this choice of $H$, $C$ and $\mathcal{C}$. We prove the following strong result.
\begin{lemma}\label{E8lemma1}
We have $H\mathcal{C}=G\setminus H$, where $H\mathcal{C}:=\{hx\mid h\in H,x\in \mathcal{C}\}$.
\end{lemma}
\begin{proof}
Let $\chi$ be an irreducible constituent of $1_H^G$ and observe that $\chi$ is unipotent. The quantity $|C_\varepsilon|\langle\chi_{|C_\varepsilon},1\rangle-\chi(1)$ can be estimated using Lemma~\ref{TZ}. Indeed, since each element of $C\setminus\{1\}$ is regular, we deduce 
$$|C|\langle\chi_{|C},1\rangle-\chi(1)=
\begin{cases}
(|C|-1)^2&\textrm{or},\\
0&\textrm{otherwise}.
\end{cases}
$$

We find
\begin{align}\label{eqE8888}
\sum_{\chi\in \mathrm{Irr}(G)}
\frac{\langle\chi,1_H^G\rangle}{\chi(1)}
\left(|C|\langle \chi_{|C},1\rangle-\chi(1)\right)^2
&\le(|C|-1)^2\left(1+\frac{1}{\chi_2(1)}+\frac{1}{\chi_3(1)}+\frac{1}{\chi_4(1)}+\frac{1}{\chi_5(1)}\right)\\\nonumber
&=(|C|-1)^2f(q),
\end{align} 
where for simplicity, we denote by $f(q)$ the factor appearing on the right hand side of~\eqref{eqE8888}.

Summing up, from~\eqref{eureka} and~\eqref{eqE8888}, we obtain
\begin{equation*}
c(\Delta(H,\mathcal{C}))
\ge
\frac{|G:H|}{f(q)}> |G:H|-1-q\cdot\frac{(q^{14}-1)(q^9+1)(q^5+1)}{(q-1)},
\end{equation*}
where the last inequality follows with a computation using the function $f(q)$ and using the fact that $|G:H|=(q^{30}-1)
(q^{12}+1)(q^{10}+1)(q^6+1)/(q-1)$.
By Lemma~\ref{l:components}~\eqref{components3} and Lemma~\ref{2F4lemma0}, we have $c(\Delta(H,\mathcal{C}))=|G:H|-1$.
\end{proof}

\begin{proposition}
The directed graph $\Gamma_2(E_8(q))$ is strongly connected.
\end{proposition}
\begin{proof}
By Theorem~\ref{theorem111}, all elements of even order of $G$ are in the same connected component of $\Gamma_1(G)$. In particular, we use the prime graph of $G$ to deduce properties of the commuting graph of $G$, see Corollary~\ref{cor}. When the prime graph of $G$ is connected, then so is $\Gamma_1(G)$ and hence so is $\Gamma_n(G)$. Therefore, we may suppose that $\Pi(G)$ is disconnected. We have reported in Table~\ref{tableE8} informations on the connected components of $\Pi(G)$. This information is taken from~\cite{KoMa}
\begin{table}[!ht]\centering
\begin{tabular}{cll}
\toprule[1.5pt]
conditions&nr. components & components\\
\midrule[1.5pt]
$q\equiv 2,3\pmod 5$&4&$\pi\left(q(q^8-1)(q^{14}-1)(q^{12}-1)(q^{18}-1)\right)$, $\pi\left(\frac{q^{10}-q^5+1}{q^2-q+1}\right)$,\\
&& $\pi\left(\frac{q^{10}+q^5+1}{q^2+q+1}\right)$, $\pi(q^8-q^4+1)$\\
$q\equiv 0,1,4\pmod 5$&5&$\pi\left(q(q^8-1)(q^{14}-1)(q^{12}-1)(q^{18}-1)\right)$, $\pi\left(\frac{q^{10}-q^5+1}{q^2-q+1}\right)$,\\
&& $\pi\left(\frac{q^{10}+q^5+1}{q^2+q+1}\right)$, $\pi(q^8-q^4+1)$, $\pi\left(\frac{q^{10}+1}{q^2+1}\right)$\\
\bottomrule[1.5pt]
\end{tabular}
\caption{Cases when $E_8(q)$ has disconnected prime graph}\label{tableE8}
\end{table}

When $q\equiv 2,3\pmod 5$, let $c\in \{
(q^{10}-q^5+1)/(q^2-q+1),(q^{10}+q^5+1)/(q^2+q+1),q^8-q^4+1\}$ and, when $q\equiv 0,1,4\pmod 5$, let 
let $c\in \{
(q^{10}-q^5+1)/(q^2-q+1),(q^{10}+q^5+1)/(q^2+q+1),q^8-q^4+1,(q^{10}+1)/(q^2+1)\}$.
 Arguing as in the other Lie type groups, Lemma~\ref{E8lemma1} guarantees that, there exists an element $g$
having order divisible by a prime in $c$ and a non-identity unipotent element  $z$ with $g \mapsto_2 z$. Moreover, if we let $C=\langle g\rangle$ be a maximal torus of $G$ with $|C|=c$, then $C<{\bf N}_G(C)$ and $|{\bf N}_G(C):C|$ is divisible by $2$. Therefore, there exists an involution $\iota\in {\bf N}_G(C)$. This gives $\iota\mapsto_2 g$ from which it follows that $\Gamma_2 (G)$ is strongly connected.
\end{proof}

\subsubsection{The exceptional group $G_2(q)$}\label{sec:G2}

\begin{proposition}
Let $q>2$ be a prime power. The directed graph $\Gamma_3(G_2(q))$ is strongly connected.
\end{proposition}
\begin{proof}
Set $G:=G_2(q)$. As usual, we may only consider the cases when $\Pi(G)$ is disconnected. These cases are reported in Table~\ref{tableG2}.
\begin{table}[!ht]\centering
\begin{tabular}{cll}
\toprule[1.5pt]
conditions&nr. components & components\\
\midrule[1.5pt]
$2<q\equiv \varepsilon\cdot 2\pmod 3$, $\varepsilon\in \{1,-1\}$&2&$\pi(q(q^2-1)(q^2-\varepsilon))$, $\pi(q^2-\varepsilon q+1)$\\
$q\equiv 0\pmod 3$&3&$\pi(q(q^2-1))$, $\pi(q^2-q+1)$, $\pi(q^2+q+1)$\\
\bottomrule[1.5pt]
\end{tabular}
\caption{Cases when $G_2(q)$ has disconnected prime graph}\label{tableG2}
\end{table}

From~\cite{bray}, we have
$$G\ge \mathrm{SL}_3(q).2 \hbox{ and }G\ge \mathrm{SU}_3(q).2.$$
From Propositions~\ref{proposition:PSL} and~\ref{proposition:PSUUU}, it follows that $\Gamma_3(\mathrm{SL}_3(q).2)$ and $\Gamma_2(\mathrm{SU}_3(q).2)$ are strongly connected and hence so is $\Gamma_3(G)$.
\end{proof}

\subsubsection{The exceptional groups $E_6(q)$ and ${}^2E_6(q)$}\label{sec:E6}

\begin{proposition}
Let $q$ be a prime power. The directed graphs $\Gamma_3(E_6(q))$ and $\Gamma_3({}^2E_6(q))$ are strongly connected.
\end{proposition}
\begin{proof}
Set $G:=E_6(q)$ or $G={}^2E_6(q)$. As usual, we may only consider the cases when $\Pi(G)$ is disconnected. These cases are reported in Table~\ref{tableE6}.
\begin{table}[!ht]\centering
\begin{tabular}{cll}
\toprule[1.5pt]
conditions&nr. components & components\\
\midrule[1.5pt]
$G=E_6(q)$&2&$\pi(q(q^5-1)(q^8-1)(q^{12}-1))$, $\pi\left(\frac{q^6+q^3+1}{\gcd(3,q-1)}\right)$\\
$G={}^2E_6(q)$, $q\ne 2$&2&$\pi(q(q^5-1)(q^8-1)(q^{12}-1))$, $\pi\left(\frac{q^6-q^3+1}{\gcd(3,q+1)}\right)$\\
$G={}^2E_6(2)$&4&$\{2,3,5,7,11\}$, $\{13\}$, $\{17\}$, $\{19\}$\\
\bottomrule[1.5pt]
\end{tabular}
\caption{Cases when $E_6(q)$ or ${}^2E_6(q)$ has disconnected prime graph}\label{tableE6}
\end{table}

The reductive subgroups of maximal rank in $E_6(q)$ and $^{2}E_6(q)$ have been investigated and classified in~\cite{LSS}. In particular, from~\cite[Table~5.1]{LSS}, we see that $E_6(q)$ contains a reductive subgroup of maximal rank having socle $\mathrm{PSL}_3(q^3)$ and ${}^2E_6(q)$ contains a reductive subgroup of maximal rank having socle $\mathrm{PSU}_3(q^3).$
From Propositions~\ref{proposition:PSL} and~\ref{proposition:PSUUU} (when $q\ne 2$), it follows that $\Gamma_3(\mathrm{PSL}_3(q^3))$ and $\Gamma_3(\mathrm{PSU}_3(q^3))$ are strongly connected and hence so is $\Gamma_2(G)$. Finally, when $G={}^2E_6(2)$, we see from~\cite{atlas} that $G$ contains subgroups isomorphic to $\mathrm{PSU}_3(8)$, $Fi_{22}$ and $\Omega_8^-(2)$. Observe that $\pi(\mathrm{PSU}_3(8))=\{2,3,7,19\}$, 
$\pi(Fi_{22})=\{2,3,5,7,11,13\}$ and $\pi(\Omega_8^-(2))=\{2,3,5,7,17\}$. It can be verified with the auxiliary help of a computer that $\Gamma_2(Fi_{22})$ and $\Gamma_2(\Omega_8^-(2))$ are strongly connected. Since $\Gamma_2(\mathrm{PSU}_3(8))$ is also strongly connected from Proposition~\ref{proposition:PSUUU}, it follows that $\Gamma_2({}^2E_6(2))$ is strongly connected.
\end{proof}

\subsubsection{The exceptional groups $E_7(q)$}\label{sec:E7}

\begin{proposition}
Let $q$ be a prime power. The directed graph $\Gamma_3(E_7(q))$ is strongly connected.
\end{proposition}
\begin{proof}
Set $G:=E_7(q)$. As usual, we may only consider the cases when $\Pi(G)$ is disconnected. These cases are reported in Table~\ref{tableE7}.
\begin{table}[!ht]\centering
\begin{tabular}{cll}
\toprule[1.5pt]
conditions&nr. components & components\\
\midrule[1.5pt]
$q=2$&3&$\{2,3,5,7,11,13,17,19,31,43\}$, $\{73\}$, $\{127\}$\\
$q=3$&3&$\{2,3,5,7,11,13,19,37,41,61,73,547\}$, $\{757\}$, $\{1093\}$\\
\bottomrule[1.5pt]
\end{tabular}
\caption{Cases when $E_7(q)$ has disconnected prime graph}\label{tableE7}
\end{table}

When $G=E_7(2)$, we see from~\cite{atlas} that $G$ contains subgroups isomorphic to $\mathrm{PSL}_8(2)$ and  $E_6(2)$. Observe that $\pi(\mathrm{PSL}_8(2))=\{2, 3, 5, 7, 17, 31, 127 \}$, 
$\pi(E_6(2))=\{2,3,5,7,13,17,31,73\}$. As $\Gamma_3(\mathrm{PSL}_8(2))$ and $\Gamma_3(E_6(2))$ are strongly connected, so is $\Gamma_3(E_7(2))$. 

When $G=E_7(2)$, using a parabolic subgroup of $G$ and using the fact that $E_6(q)<E_7(q)$, we see that $G$ contains subgroups isomorphic to $\mathrm{PSL}_7(3)$ and  $E_6(3)$. Observe that $\pi(\mathrm{PSL}_7(3))=\{
  2, 3, 5, 7, 11, 13, 1093\}$ and  
$\pi(E_6(3))=\{ 2, 3, 5, 7, 11, 13, 41, 73, 757\}$. As $\Gamma_3(\mathrm{PSL}_7(3))$ and $\Gamma_3(E_6(3))$ are strongly connected, so is $\Gamma_3(E_7(3))$. 
\end{proof}

\subsection{Almost simple groups having socle a Suzuki group ${}^2B_2(q)$}\label{Suzuki}
Let $q:=2^f$ and let $G$ be an almost simple group having socle the Suzuki group $L:={}^2B_2(q)$. Since the Engel graph of $L$ is not strongly connected, in this section we suppose that $G>L$.

Now, $G=L\rtimes\langle\alpha\rangle$, where $\alpha$ is a field automorphism. Let $p$ be a prime with $p$ dividing the order of $\alpha$. Set $\beta:=\alpha^{{\bf o}(\alpha)/p}$, $q_0:=q^{1/p}=2^{f/p}$ and $f':=f/p$. 

We claim that any two elements of $L$ having order $2$ are in the same strongly connected component of $L$. Let $x,y\in L$ with ${\bf o}(x)={\bf o}(y)=2$. Let $P$ be a Sylow $2$-subgroup of $L$ with $x\in P$. If $y\in P$, then $[x,y,y]=1$ because $P$ has nilpotency class $2$. Therefore, suppose $y\notin P$. Let $P^-$ be the opposite Sylow $2$-subgroup of $L$ (here we are thinking of fixing a root system for the Lie group ${}^2B_2(q)$). Now, $P$ acts transitively on the set of Sylow $2$-subgroups of $L$ distinct from $P$. Therefore, replacing $y$ and $x$ with a suitable $P$-conjugate, we may suppose that $y\in P^-$. 
Replacing the field authomorphism $\beta$ with a suitable $L$-conjugate if necessary, we may suppose that $Q:=P\cap {\bf C}_L(\beta)$ is a Sylow $2$-subgroup of ${\bf C}_L(\beta)={}^2B_2(q_0)$. In particular, $Q^-=P^-\cap {\bf C}_L(\beta)$ is the opposite Sylow $2$-subgroup of ${\bf C}_L(\beta)={}^2B_2(q_0)$. Now, let $x_0\in Q\setminus \{1\}$ and $y_0\in Q^-\setminus\{1\}$. We have $[x,x_0,x_0]=1$, $[x_0,\beta]=1$, $[\beta,y_0]=1$ and $[y_0,y,y]=1$, that is,
$$x\mapsto_2 x_0\mapsto_1\beta\mapsto_1 y_0\mapsto_2y,$$ 
which is what we wanted to prove. 

Let $\Omega$ be a strongly connected component of $\Gamma_2(G)$ containing an element of order $2$. From the paragraph above, $\Omega$ contains all the elements having even order and contains all the $L$-conjugates of $\beta$.

Now, when $p\equiv 1,7\pmod 8$, we have that $q_0+\sqrt{2q_0}+1$ divides $q+\sqrt{2q}+1$ and $q_0-\sqrt{2q_0}+1$ divides $q-\sqrt{2q}+1$. Similarly,  when $p\equiv 3,5\pmod 8$, we have that $q_0+\sqrt{2q_0}+1$ divides $q-\sqrt{2q}+1$ and $q_0-\sqrt{2q_0}+1$ divides $q+\sqrt{2q}+1$.  Observe again that the centralizer of $\beta$ in $L$ is isomorphic to ${}^2B_2(q^{1/p})={}^2B_2(q_0)$.

From these facts, it is not hard to deduce  that $\Gamma_2(G)$ is strongly connected, except when $q_0-\sqrt{2q_0}+1=1$. Indeed, when $q_0-\sqrt{2q_0}+1=1$ and $p\equiv 1,7\pmod 8$, we cannot guarantee that we reach the elements of order $q+\sqrt{2q}+1$ in $L$ from elements of ${\bf C}_L(\beta)$. And similarly, when $q_0-\sqrt{2q_0}+1=1$ and $p\equiv 3,5\pmod 8$, we cannot guarantee that we reach the elements of order $q-\sqrt{2q}-1$ in $L$ from elements of ${\bf C}_L(\beta)$.

Suppose $q_0-\sqrt{2q_0}+1=1$, that is, $q_0=2$ and hence $f=p$ is a prime number. Let $T$ be a maximal torus of $L$ of order $q+\sqrt{2q}+1$, let $y\in T$ having order $q+\sqrt{2q}+1$ and let $K:={\bf N}_G(T)$. Cleary, $K$ is a maximal subgroup of $G$. In particular, ${\bf N}_G(K)=K$ and~\eqref{eq:NC1} in Lemma~\ref{NC} is satisfied. It is not hard to verify that also~\eqref{eq:NC2} in Lemma~\ref{NC} is satisfied. Thus the Engel graph of $G$ is disconnected.
\end{document}